\documentclass[11pt,reqno]{amsart}
\usepackage[top=1in, bottom=1in, left=1.3in, right=1.3in]{geometry}
\usepackage{amsmath}
\usepackage{amssymb}
\usepackage{amsthm}
\usepackage{verbatim}
\usepackage{graphicx}
\usepackage{graphics}
\usepackage{enumerate}
\usepackage{mathrsfs}
\usepackage{booktabs}
\usepackage[toc,page]{appendix}
\usepackage{enumerate}
\usepackage{mathrsfs}
\usepackage{fancyhdr}
\usepackage{latexsym, amsfonts, amssymb, amsmath,  amsthm}
\usepackage{amsopn, amstext, amscd,pifont}
\usepackage{amssymb,bm, cite}
\usepackage[all]{xy}
\usepackage{color}
\usepackage{blkarray}
\usepackage{graphicx,hyperref}
\usepackage{float}

\usepackage[utf8]{inputenc}
\usepackage{cancel}

\theoremstyle{plain}

\newtheorem{theorem}{Theorem}[section]
\newtheorem{proposition}[theorem]{Proposition}
\newtheorem{lemma}[theorem]{Lemma}
\newtheorem{corollary}[theorem]{Corollary}
\newtheorem{question}[theorem]{Question}

\theoremstyle{definition}
\newtheorem{definition}[theorem]{Definition}
\newtheorem{example}[theorem]{Example}

\newtheorem{claim}{Claim}
\newtheorem*{acknowledgements}{Acknowledgements}
\newtheorem*{outline}{Outline of the paper}

\theoremstyle{remark}

\newtheorem{remark}[theorem]{Remark}

\def\Z{\mathbb{Z}}
\def\C{\mathbb{C}}

\def\Q{\mathbb{Q}}

\def\P{\mathbb{P}}

\def\Z2{\mathbb{Z}_{2}}
\def\m2#1{\ ({\rm mod} \ 2^{#1})}

\makeatletter
\@namedef{subjclassname@2020}{%
  \textup{2020} Mathematics Subject Classification}
\makeatother

\newcommand{\hmn}[1]{\textcolor{blue}{\sf #1}}

\begin{document}

\title[$p$-adic sub-hyperbolic rational maps]{Julia sets and dynamics of $p$-adic sub-hyperbolic rational maps} 

\author{Shilei Fan}
\address{School of Mathematics and Statistics, and Key Laboratory of Nonlinear Analysis \& Applications (Ministry of Education), Central China Normal University, Wuhan 430079, P. R. China}  \email{slfan@mail.ccnu.edu.cn}

\author{Lingmin Liao}
\address{School of Mathematics and Statistics, Wuhan University, Wuhan 430072, Hubei, China}
\email{lmliao@whu.edu.cn}

\author{Hongming Nie}
\address{Institute for Mathematical Sciences, Stony Brook University, Stony Brook, NY 11794-3660, USA}
\email{hongming.nie@stonybrook.edu}

\author{Yuefei Wang}
\address{College of Mathematics and Statistics, Shenzhen University, Shenzhen 518060, Guangdong, China \& Academy of Mathematics and System Sciences, CAS, Beijing 100190, China}
\email{wangyf@math.ac.cn}

\thanks{S. L. FAN was partially supported by NSFC (grants No. 12331004  and No. 11971190) and Fok Ying-Tong Education Foundation, China (grant No.171001).  Y. F. WANG was partially supported by NSFC (grants No. 12231013) and NCAMS}


\begin{abstract}
Let  $K$ be a finite extension of the field $\mathbb{Q}_p$ of $p$-adic numbers, and  $\phi\in K(z)$ be a rational map of degree $d\ge 2$. We prove that  if $\phi$ is sub-hyperbolic then the dynamics of $\phi$ restricted on its $K$-Julia set is a countable state Markov shift.  We give a complete characterization of the minimal decomposition for the $K$-Fatou set of $\phi$. Moreover, we prove that if $\phi$ is contained in certain conjecturally dense subset, properly containing semi-hyperbolic rational maps, in $K(z)$, then the $K$-Julia set of $\phi$ is the natural restriction of $\mathbb{C}_p$-Julia set.
\end{abstract}


\maketitle


\medskip

\tableofcontents

\section{Introduction}\label{sec:intro}

\subsection{Motivation and background}

In recent decades, $p$-adic and non-archimedean dynamics have been attracting much attention. On the one hand, many $p$-adic analogues of classical complex dynamics results already appear in the literature. A Siegel's linearization theorem for non-Archimedean fields has been presented in a paper of Herman and Yoccoz \cite{Herman-Yoccoz}, and a non-Archimedean Montel's theorem appears in a work of Hsia \cite{Hsia00} (see also a paper of Favre, Kiwi, and Trucco \cite{Favre12}).
Some other significant contributions in this direction arise in the Ph.D. theses of Benedetto \cite{Benedetto98}  and Rivera-Letelier \cite{RL2003Ast} and in Kiwi's descriptions \cite{Kiwi06, Kiwi14} of the dynamics for lower degree rational maps.  On the other hand, the ergodic aspect of the aforementioned dynamics develops intensively. The ergodic decomposition starts with an early result of Oselies and Zieschang \cite{Oselies75}, and then the $p$-adic entropy shows up in a study of Lind and Ward \cite{Lind88}. Later, some ergodic criteria for uniform continuous functions gather in a paper of Anashin \cite{Anashin94}. The equilibrium measure on the related Berkovich projective line emerges in works of  Baker and Rumely \cite{Baker06}, Chambert-Loir \cite{Chambert06}, and Favre and Rivera-Letelier \cite{Favre10}.

For a fixed prime number $p\ge 2$, denote by $\mathbb{Q}_p$ the field of $p$-adic numbers and by $\mathbb{C}_p$ the completion of an algebraic closure of $\mathbb{Q}_p$. Let $K$ be a finite extension of $\mathbb{Q}_p$.  Write $\mathbb{P}^1_K$ and $\mathbb{P}^1_{\mathbb{C}_p}$ the projective lines over $K$ and $\mathbb{C}_p$, respectively. In this paper, we study rational maps $\phi\in K(z)$ as dynamical systems on  $\mathbb{P}^1_K$ and on $\mathbb{P}^1_{\C_p}$.

The rational map $\phi$ partitions naturally the space $\mathbb{P}^1_{\mathbb{C}_p}$ into two parts: the Fatou set $F_{\C_p}(\phi)$ on which the iterations $\phi^n$ act equicontinuously and the Julia set $J_{\C_p}(\phi)$ (the complement of $F_{\C_p}(\phi)$) on which the iterations $\phi^n$ act chaotically. One main task in non-archimedean dynamics is to depict these two sets and further understand the relevant dynamics. For $F_{\C_p}(\phi)$, Benedetto \cite{Benedetto00} proved that there are no wandering domains  under a no wild recurrent Julia critical points assumption,  and  Rivera-Letelier \cite{RL2003Ast} gave a classification of the Fatou components.
For $J_{\C_p}(\phi)$, applying  aforementioned Montel's theorem, Hsia \cite{Hsia00} obtained that $J_{\C_p}(\phi)$ is contained in the closure of the periodic points. Then assuming $\phi$ has a repelling periodic point, B\'ezivin \cite{Bezivin01} proved that $J_{\C_p}(\phi)$ is the closure of the repelling periodic points.


In $\P^1_K$, the Julia set $J_K(\phi)$ is the region on which the sequence $\{\phi^n\}_{n\geq 1}$ is not equicontinuous with respect to the spherical metric. Although $J_K(\phi)$ is a subset of $J_{\C_p}(\phi)\cap\mathbb{P}^1_K$ by definition, the equality of these two sets is elusive.
Nevertheless, all periodic points in $J_{\C_p}(\phi)$ are repelling (e.g., see \cite[Proposition 1.1]{Benedetto01}), which rules out a  severe obstruction to preserve the consistency of the Julia points occurring in complex/real dynamics.  (A precise example in the latter is the polynomial $P(z)=z^3-z\in\mathbb{R}[z]$ acting on $\mathbb{P}^1_\mathbb{C}$ and $\mathbb{P}^1_\mathbb{R}$, respectively, for which the point $0$ is a parabolic fixed point in $\mathbb{P}^1_\mathbb{C}$ attracting all the nearby points in $\mathbb{P}^1_\mathbb{R}$ and hence $0$ is contained in the $\mathbb{C}$-Julia set but not in the $\mathbb{R}$-Julia set of $P$.) 
It is interesting to ask the following question.
\begin{question}\label{question}
For a rational map  $\phi\in K(z)$, is $J_{\mathbb{C}_p}(\phi)\cap\mathbb{P}^1_K=J_K(\phi)$?
\end{question}

If $\phi$ has degree strictly less than $2$, it follows immediately from the simple behavior of the iterates $\phi^n$ (see \cite[Section 1.5]{Benedetto19}) that $J_{\mathbb{C}_p}(\phi)\cap\mathbb{P}^1_K=J_K(\phi)$. In the case that $\phi$ is a tame polynomial of degree at least $2$, assuming that all (Berkovich) Julia points possess finite algebraic degree, Trucco \cite[Theorem E and Remark 1.2]{Trucco14} proved that $J_{\mathbb{C}_p}(\phi)$ has no recurrent critical points and is contained in a finite extension of $K$.
 In this paper, we will first provide an affirmative answer of Question \ref{question} for a large and conjecturally dense family of rational maps, see Theorems \ref{Thm:main1} and \ref{Thm:main}.


Once the Fatou and Julia sets are at our disposal, we can attempt to investigate the dynamical behaviors on such sets. In general, we may understand the dynamics by showing some fundamental properties such as the ergodicity and the linearization, but a detailed description of the behaviors for orbits seems improbable. However, this is not the case when we study the dynamics of $\phi\in K(z)$ on $\mathbb{P}^1_K$, for which, in many cases, we can employ well-known dynamical models to characterize the Fatou and Julia dynamics of $\phi$.

For the dynamics on Fatou set, Fan and the second author  \cite{Fan11} proved a minimal decomposition result for the polynomials defined over $\mathbb{Z}_p$, and showed that the dynamics on each minimal component is topologically conjugate to an odometer.  Such a decomposition also appears for other maps such as convergent power series on the integral ring of a finite extension of $\mathbb{Q}_p$  \cite{Fan15}, rational maps over $\mathbb{Q}_p$ of degree one  \cite{Fan14}, and rational maps over $\mathbb{Q}_p$ having good reduction  \cite{Fan17}. The first and second authors provided concrete examples of the minimal decompositions for the Chebyshev polynomials on $\mathbb{Z}_2$ and the square maps on $\mathbb{Z}_p$ in \cite{Fan16J} and \cite{Fan16I}, respectively.

For the dynamics on Julia set, a classical example is the polynomial $(z^p-z)/p$ on $\mathbb{Z}_p$, which turns out to be topologically conjugate to the full shift on the symbolic space with $p$ symbols, see \cite{Woodcock98}. This property also holds for quadratic polynomials over $\mathbb{Q}_p$, see \cite{Thiran89}.  In the paper \cite{Fan07}, the authors proved that any $p$-adic transitive weak repeller is topologically conjugate to a subshift of finite type on an alphabet of finitely many symbols. Such $p$-adic transitive weak repellers include expanding polynomials and rational maps over $\mathbb{Q}_p$. We refer to \cite{Fan18, Fan16} for the  concrete example $az+1/z$, and to \cite{Albeverio13,Khakimov21} for some other examples.

In this paper, we will also study the dynamics on the Julia set for more general rational maps which are out of range of the previous work in \cite{Fan07}. We will prove in Theorem \ref{Thm:GF} that such a dynamical system can be modeled by a countable state Markov shift.

\subsection{Main results}
Recall that $K$ is a finite extension of $\mathbb{Q}_p$. Let $\phi\in K(z)$ be a rational map of degree at least $2$. A point $x\in\mathbb{P}^1_{\mathbb{C}_p}$ is {\it eventually periodic} under $\phi$ if there exist $m\ge 0$ and $n\ge 1$ such that $\phi^{m+n}(x)=\phi^m(x)$. A point $x\in\mathbb{P}^1_{\mathbb{C}_p}$ is {\it non-recurrent} under $\phi$ if $x$ is not accumulated by its forward orbit under $\phi$; otherwise, we say $x$ is {\it recurrent}.

Following Benedetto \cite[Main Theorem]{Benedetto01},  we say that $\phi$ is {\it hyperbolic} if the $\mathbb{C}_p$-Julia set $J_{\mathbb{C}_p}(\phi)$ contains no critical points of $\phi$. As in complex dynamics \cite[Section 19]{Milnor06}, we can generalize the above hyperbolicity to sub-hyperbolicity and semi-hyperbolicity as following.

\begin{definition}
Let $\phi\in K(z)$ be a rational map of degree at least $2$.
We say $\phi$ is {\it sub-hyperbolic} if every critical point of $\phi$ in $J_{\mathbb{C}_p}(\phi)$ is eventually periodic; and we say $\phi$ is {\it semi-hyperbolic} if every critical point of $\phi$ in $J_{\mathbb{C}_p}(\phi)$ is non-recurrent.
\end{definition}
Observe that if $\phi$ is hyperbolic, then it is sub-hyperbolic; and if $\phi$ is sub-hyperbolic, then it is semi-hyperbolic.  Definitely there exist rational maps in $K(z)$ which are not semi-hyperbolic, but it is worth mentioning that the hyperbolic maps and hence the sub-hyperbolic/semi-hyperbolic maps are conjectured to be dense, see \cite{Benedetto01}. We will study the (singular) metric properties, such as expansion near Julia sets (away singularities),  of sub-hyperbolic/semi-hyperbolic rational maps in a companion paper.

\smallskip

In this present paper,  our first result provides an affirmative answer of Question \ref{question} for the semi-hyperbolic rational maps in $K(z)$. 

\begin{theorem}\label{Thm:main1}
Let $\phi\in K(z)$ be a semi-hyperbolic rational map of degree at least $2$. Then
$$J_{\mathbb{C}_p}(\phi)\cap\mathbb{P}^1_K=J_K(\phi).$$
\end{theorem}

Although we are not able to completely answer Question \ref{question} here, we in fact show an affirmative answer for a large family of rational maps with possibly certain recurrent critical points which properly contains the semi-hyperbolic rational maps. To avoid the technical assumptions, for now we just state it for the semi-hyperbolic rational maps as in the above result; and we will present the full version of the result in Section \ref{subsec:relation}.

\smallskip
Notice that for a semi-hyperbolic rational map, any Julia critical orbit may have infinitely many accumulated points; while  for a sub-hyperbolic rational map, any Julia critical orbit is finite. Such finiteness provides a way to encode the Julia dynamics of sub-hyperbolic rational maps with countably many symbols, which will be stated in our next result. 

 For a sub-hyperbolic rational map $\phi\in K(z)$, write $\mathrm{Crit}_K^\ast(\phi)$ the set of the critical points in $J_{\mathbb{C}_p}(\phi)\cap\mathbb{P}^1_K$ and denote by $\mathrm{GO}_K(\mathrm{Crit}_K^\ast(\phi))$ the grand orbit of $\mathrm{Crit}_K^\ast(\phi)$ in $\mathbb{P}^1_K$, that is, the union of grand orbits of critical points in $\mathrm{Crit}_K^\ast(\phi)$. By Theorem \ref{Thm:main1} and the invariance of $J_{\mathbb{C}_p}(\phi)$, we conclude that $\mathrm{GO}_K(\mathrm{Crit}_K^\ast(\phi))\subset J_K(\phi)$. Set
$$I_K(\phi):=J_K(\phi) \setminus \mathrm{GO}_K(\mathrm{Crit}_K^\ast(\phi)).$$

\begin{theorem}\label{Thm:GF}
Let $\phi\in K(z)$ be a sub-hyperbolic rational map of degree at least $2$. Then there exist a countable state Markov shift $(\Sigma_A,\sigma)$ and a bijection $h: J_K(\phi)\to\Sigma_A$ such that $(I_K(\phi),\phi)$ is topologically conjugate to $(h(I_K(\phi)), \sigma)$ via $h$.
\end{theorem}

If $\mathrm{Crit}_K^\ast(\phi)=\emptyset$, then the shift space $\Sigma_A$ in Theorem \ref{Thm:GF} can be chosen with finite states, which is a special case of a result in \cite{Fan07} for $p$-adic weak repellers. If $\mathrm{Crit}_K^\ast(\phi)\not=\emptyset$, the presence of a critical point in $J_K(\phi)$ implies that $\phi$ is not uniformly expanding on $J_K(\phi)$. To construct the states of the Markov shift, we track the orbits of countably many open disks near critical points. To illustrate Theorem \ref{Thm:GF}, in Section \ref{sec:example} we provide a concrete example whose Gurevich entropy on Julia set is the logarithm of an algebraic number.

\smallskip
Let us make the following remark on potential applications of Theorem \ref{Thm:GF}.
\begin{remark}
The symbolic dynamics in Theorem \ref{Thm:GF}  provides a way to study the ergodic properties for sub-hyperbolic rational maps in $K(z)$ acting on $\mathbb{P}^1_K$.  The associated matrix $A$ in Theorem \ref{Thm:GF} allows us to find all transitive subsystems of the given sub-hyperbolic rational map. If the matrix $A$ is explicit, then it is possible to compute the entropy of such subsystems, see Section \ref{sec:Gurevich} for the computations of a concrete example. Moreover, with this associated matrix, it is also possible to determine if on such a subsystem there exists a unique measure of maximal entropy and further if this maximal entropy measure is exponentially mixing. For more details, see Section \ref{sec:applications}. We refer to \cite{Favre12} for some ergodic properties of rational maps acting on $\mathbb{P}^1_{\mathbb{C}_p}$ and on the corresponding Berkovich space.
\end{remark}


With the aid of Theorem \ref{Thm:GF}, we can further characterize the global dynamics for sub-hyperbolic rational maps. We say that a system is {\it minimal} if every orbit is dense. 

\begin{theorem}\label{globaldynamics}
Let $\phi\in K(z)$ be  a sub-hyperbolic rational map of degree at least $2$.  Then $\P^{1}_K$  has the following decomposition 
\begin{align}\label{decomposition}
\P^{1}_K= P\sqcup M\sqcup B \sqcup J 
\end{align}
where 
\begin{enumerate} 
\item $P$ is a finite set consisting of all non-repelling periodic points of
$\phi$, 
\item $ M = \bigsqcup_i M _i$ is the union of all (at most countably
many) clopen invariant sets such that each $M_i$ is a finite union
of $\P^1_K$-disk  and each subsystem $\phi: M_i \to M_i$ can be decomposed to minimal subsystems, 
\item  $B$ is the set of points in the basin of a periodic orbit or of a subsystem $\phi: M_i \to M_i$, and 
 \item $J=J_K(\phi)$ is the Julia set of $\phi$ on which the dynamical behavior can be characterized by a countable state Markov shift. 
\end{enumerate} 
\end{theorem}

In the above result, if $K=\mathbb{Q}_p$, the each subsystems $\phi: M_i \to M_i$ is minimal. If $K\not=\mathbb{Q}_p$, then there are uncountably many minimal subsystems in each system $\phi: M_i \to M_i$, see  \cite[Theorem 1.1]{Fan15}. As we will see in the proof, the statements (1)-(3) in fact hold for general rational maps in $K(z)$ possessing  no wondering domains and satisfying the conclusion of Theorem \ref{Thm:main1}, but the statement (4) relies on Theorem \ref{Thm:GF}. 

\begin{remark}
In our arguments of Theorems \ref{Thm:GF} and \ref{globaldynamics}, we in fact work on the rational maps $\phi\in K(z)$ whose all critical points  in $J_{\mathbb{C}_p}(\phi)\cap \mathbb{P}^1_K$ are  eventually periodic. Such rational maps may be slightly more general than the sub-hyperbolic ones, since a $\mathbb{C}_p$-Julia critical point may not contain in $\mathbb{P}^1_K$. Definitely, for a given rational map $\psi\in K(z)$, considering some finite extension  $L$ of $K$ so that  all critical points of $\psi$ are contained in $\mathbb{P}^1_L$, we have that $\psi$ is sub-hyperbolic if all critical points  in $J_{\mathbb{C}_p}(\psi)\cap \mathbb{P}^1_L$ are  eventually periodic.
\end{remark}

To end this subsection,  we  remark here that our results can be extended to some $\mathbb{C}_p$-Julia sets.
\begin{remark}
As aforementioned, there exist polynomials in $\mathbb{C}_p[z]$ whose $\mathbb{C}_p$-Julia sets are contained in some finite extensions of $\mathbb{Q}_p$, see \cite[Theorem E and Remark 1.2]{Trucco14}. Thus for such polynomials, under the sub-hyperbolicity assumption, the conclusions in Theorems \ref{Thm:GF} and \ref{globaldynamics} also hold for the corresponding $\mathbb{C}_p$-Julia sets.
\end{remark}

\subsection{Relation of the Julia sets}\label{subsec:relation}
We now state our general result in the process of answering Question \ref{question}, which immediately implies Theorem \ref{Thm:main1}. 
 
 A critical point $c\in\mathbb{P}^1_{\mathbb{C}_p}$ of a rational map $\phi\in K(z)$ is  {\it wild} if $p$ divides the local degree of $\phi$ at $c$. Otherwise, we say that $c$ is {\it tame}.  
As we will see in Section \ref{sec:equivalence}, the inclusion of the closures of forward orbits induces an equivalence relation on the critical points in $J_{\mathbb{C}_p}(\phi)$ together with a partial order $\preceq$ on the resulting quotient set $\mathcal{M}(\phi)$. Each least element in $(\mathcal{M}(\phi),\preceq)$ is called a {\it minimal class}.


\begin{theorem}\label{Thm:main}
Let $\phi\in K(z)$ be a rational map of degree at least $2$. Suppose that
\begin{enumerate}
\item $\phi$ has no wild recurrent critical points in $J_{\mathbb{C}_p}(\phi)\cap\mathbb{P}^1_K$, and
\item each minimal class in $(\mathcal{M}(\phi),\preceq)$ has at most two representatives in $J_{\mathbb{C}_p}(\phi)\cap\mathbb{P}^1_K$.
\end{enumerate}

Then
$$J_{\mathbb{C}_p}(\phi)\cap\mathbb{P}^1_K=J_K(\phi).$$
\end{theorem}

As aforementioned, one inclusion $J_K(\phi)\subset J_{\mathbb{C}_p}(\phi)\cap\mathbb{P}^1_K$ in Theorem \ref{Thm:main} follows unconditionlly from the definitions. Our contribution is to show  the reverse inclusion. The difficulty arises from the complexity of critical orbits in $J_{\mathbb{C}_p}(\phi)\cap\mathbb{P}^1_K$. Even under the assumptions in Theorem \ref{Thm:main}, the argument is far from straightforward. We provide an elaborate analysis on the orbits of disks near the critical points, and deduce the conclusion, in most cases, by an argument on the existence of repelling periodic points.

Let us emphasis on the assumptions in Theorem \ref{Thm:main}. For the first assumption, the absence of wild recurrent critical points in $J_{\C_p}(\phi)\cap\mathbb{P}^1_K$ guarantees a no wandering domains result and controls the ratios of diameters of certain disks.  We should mention here that there are rational maps possessing wild recurrent critical points in Julia set, see \cite{Rivera05}.  The second assumption is a technical one. Though most of analysis works for general case, the argument concerning the existence of repelling periodic points requires the restriction on the number of representatives for the minimal classes. Loosely speaking, for each disk in a fixed sphere close to a critical point, such a restriction bounds uniformly the ``location" (the first critical point covered by the orbit of the disk) and the ``time" (the number of iteration to cover the {first} critical point). Without this restriction, the ``location" and the ``time" may vary simultaneously, which is excluded from our arguments.

\begin{proof}[Proof of Theorem \ref{Thm:main1} assuming Theorem \ref{Thm:main}]
It suffices to check that the semi-hyperbolic rational maps satisfy the assumptions in Theorem \ref{Thm:main}. The semi-hyperbolicity implies that $\phi$ has no recurrent critical points in $J_{\mathbb{C}_p}(\phi)$; so $\phi$ satisfies the assumption (1). Moreover,  if $J_{\mathbb{C}_p}(\phi)\cap\mathbb{P}^1_K$ is nonempty, then each minimal class in $(\mathcal{M}(\phi),\preceq)$ has only one representative in $J_{\mathbb{C}_p}(\phi)\cap\mathbb{P}^1_K$. Indeed, suppose otherwise. By the definition of minimal class in Section \ref{sec:equivalence}, there exist critical points $c_0,c_1$ in $J_{\mathbb{C}_p}(\phi)\cap\mathbb{P}^1_K$ such that for any $i\in\{0,1\}$, the critical point $c_i$ is contained in the closure of the forward orbit of $c_j$, where $j=1-i\in\{0,1\}$. It follows that $c_i$ is recurrent for each $i\in\{0,1\}$, which contradicts the semi-hyperbolicity assumption on $\phi$. Thus $\phi$ also satisfies the assumption (2).
 \end{proof}


\subsection{Strategy of the proofs}
The proof of Theorems \ref{Thm:GF} is a standard argument on the decomprosition of the Julia set. Let us make it more precise. Applying the locally scaling properties of $\phi$, we obtain a cover $\mathcal{P}_0$ of $J_K(\phi)$, each of whose elements is either an open disk on which $\phi$ is bijective, or a critical point. Decompose $\mathcal{P}_0$ along the critical orbits and  divide some ``larger" disks into ``smaller" ones so that in the resulting cover $\mathcal{P}$, an iterated  image of any open set is a union of elements in $\mathcal{P}$. Considering the images of elements in $\mathcal{P}$, we obtain a natural (infinite) matrix, and then assign each point in $J_K(\phi)$ a code sequence to obtain a coding map $h$. By the inverse of $\phi$ in each open set in $\mathcal{P}$, we deduce the desired conjugacy via $h$. In particular, in the argument showing the bijectivity of $h$, we apply Theorem \ref{Thm:main1} to treat the grand orbit of critical points.
Once Theorem \ref{Thm:GF} is in our disposal, Theorem \ref{globaldynamics} is a consequence of \cite[Theorem 1.1]{Fan15}.

\smallskip
Now we briefly state the sketch of proof for the inclusion $J_{\mathbb{C}_p}(\phi)\cap\mathbb{P}^1_K\subset J_K(\phi)$ in Theorem \ref{Thm:main}. We establish this inclusion in the following three main steps, in which the last step is the most technical part in the paper.

First, we work on the points on the boundary of $J_{\C_p}(\phi)\cap\P^1_K$ (see Lemma \ref{boundary-julia-lem}). The no wandering domains result (Theorem \ref{no-Wandering-thm}) implies that any point in $\P^1_K$ close to such a boundary point is eventually mapped into a periodic component of $F_{\C_p}(\phi)$, which breaks the equicontinunity.

Second, for the interior of $J_{\mathbb{C}_p}(\phi)\cap\mathbb{P}^1_K$, we reduce to discuss the critical points in $J_{\mathbb{C}_p}(\phi)\cap\mathbb{P}^1_K$ (see Proposition \ref{lem:critical-Julia}). In this process, we apply a technical lemma (Lemma \ref{lem:uniform-scaling}) on the orbit of a disk in $\mathbb{P}^1_K$ intersecting $J_{\mathbb{C}_p}(\phi)\cap\mathbb{P}^1_K$, where the absence of wild recurrent critical points controls the ratios of diameters of certain disks.

Last, we show that the critical points in $J_{\mathbb{C}_p}(\phi)\cap\mathbb{P}^1_K$ are indeed contained in $J_K(\phi)$. To achieve this, we consider natural equivalence classes of critical points in $J_{\mathbb{C}_p}(\phi)$ and further reduce our investigation to that on certain minimal classes (see Proposition \ref{prop:minimal}).
For each such minimal class, under the assumption on the number of representatives, specially in the recurrent case, we prove two crucial results (Propositions \ref{prop:repelling} and \ref{prop:repelling-2}) concerning the existence of repelling (pre)periodic points in desired disks.
We then proceed the proof by contradiction. If there is a point contained in $J_{\mathbb{C}_p}(\phi)\cap\mathbb{P}^1_K$ but not in $J_K(\phi)$, we verify the assumptions in the aforementioned results and obtain contradictions (see Corollaries \ref{coro:1} and \ref{coro:2}).



\begin{outline}
In Section \ref{sec:pre}, we give some relevant preliminaries. In particular, we describe the local dynamics near critical points (see Proposition \ref{prop:tame-wild}) in Section \ref{sec:analytic}, and we state an invariance property of the Fatou set (see Lemma \ref{regular-point-lemma}) and a no wandering domains result (see Theorem \ref{no-Wandering-thm}) in Sections \ref{sec:Julia-Fatou} and \ref{sec:Fatou-components}.

In Section \ref{restriction}, we prove Theorem \ref{Thm:main}. We first deal with the boundary of  $J_{\mathbb{C}_p}(\phi)\cap\mathbb{P}^{1}_{K}$ in Section \ref{sec:boundary}. Then we consider the interior of  $J_{\mathbb{C}_p}(\phi)\cap\mathbb{P}^{1}_{K}$ in Section \ref{sec:reduce}. After that, we focus on some minimal equivalence classes in Section \ref{sec:equivalence}. Finally in Section \ref{sec:minimal}, we establish Propositions \ref{prop:repelling} and \ref{prop:repelling-2} on the existence of repelling periodic points. In particular, Section \ref{sec:recurrent} is the most technical part, which concerns the recurrent critical points.

Sections \ref{sec:example} and \ref{sec:GF-maps} are devoted to describing the dynamics for sub-hyperbolic rational maps. Section \ref{sec:example} provides a motivation example and contains numerous computations. In particular, Section \ref{sec:Gurevich} gives the Gurevich entropy for this example. Section \ref{sec:GF-maps} covers the proof of Theorems \ref{Thm:GF} and \ref{globaldynamics}. Section \ref{sec:reduced} states a reduced version of Theorem \ref{Thm:GF} and then Section \ref{sec:general} finishes the full general version. Section \ref{sec:applications} lists (potential) applications of Theorem \ref{Thm:GF}. Finally, in Section \ref{sec:pf-global}, we prove Theorem \ref{globaldynamics}.
\end{outline}

\begin{acknowledgements}
The authors are grateful to the Department of Mathematics, Zhejiang University, and Shanghai Center for Mathematical Sciences, where part of this paper was written, for their hospitality. They thank Robert Benedetto for discussing the no wandering domains result and also thank Liang-Chung Hsia for useful comments.
\end{acknowledgements}

\section{Preliminaries}\label{sec:pre}

In this section, we introduce some preliminaries that we will use in the rest of the paper.

\subsection{Notation}
We set the following notation.

\begin{tabular}{cp{\textwidth}}
$K$ &  a finite extension of $\Q_p$\\
$\pi$ & a uniformizer  of $K$\\
$L$  &the field $K$ or $\C_p$\\
$\mathrm{Crit}_{\mathbb{C}_p}(\phi)$ &the set of the critical points of $\phi$ in $\mathbb{P}^{1}_{\mathbb{C}_p}$\\
$\mathrm{Crit}_K(\phi)$ &the set of the critical points of $\phi$ in $\P^1_K$\\
 $\mathrm{Crit}^\ast_{\C_p}(\phi)$ &the intersection $J_{\C_p}(\phi)\cap\mathrm{Crit}_{\C_p}(\phi)$\\
 $\mathrm{Crit}^\ast_K(\phi)$ &the intersection $J_{\C_p}(\phi)\cap\mathrm{Crit}_K(\phi)$\\
$J_{\mathbb{C}_p}^{K}(\phi)$ &the intersection $J_{\mathbb{C}_p}(\phi)\cap\mathbb{P}^{1}_{K}$\\
$F_{\mathbb{C}_p}^{K}(\phi)$ &the intersection $F_{\mathbb{C}_p}(\phi)\cap\mathbb{P}^{1}_{K}$\\
$\mathcal{O}_\phi(x)$ &the forward orbit $\{\phi^n(x)\}_{n\ge 0}$ of $x\in\P^1_{\C_p}$ under a rational map $\phi\in\C_p(z)$\\
$\overline{\mathcal{O}_\phi(x)}$ &the closure of $\mathcal{O}_\phi(x)$ for a point $x\in\P^1_{\C_p}$ and a rational map $\phi\in\C_p(z)$.\\

\end{tabular}

\subsection{$p$-adic disks}\label{sec:fields}
Let $|\cdot|_p$ be the natural and non-trivial $p$-adic absolute value associated to $L$.
For $x\in L$ and $r>0$, define
$$D_L(x,r):=\{z\in L: |z-x|_p<r\} \ \text{and}\ \overline{D}_L(x,r):=\{z\in L: |z-x|_p\le r\}.$$
We call $D_L(x,r)$ (resp. $\overline{D}_L(x,r)$) an open (resp. closed) $L$-disk.
In the projective space $\mathbb{P}^{1}_L$, an \textit{open $\mathbb{P}^1_L$-disk} is a set of form either $D_L(z,r)$ or $\mathbb{P}^1_L \setminus \overline{D}_L(z,r)$, and a \textit{closed $\mathbb{P}^1_L$-disk} is a set of form either $\overline{D}_L(z,r)$ or $\mathbb{P}^1_L \setminus D_L(z,r)$.

 Regarding $\mathbb{P}^{1}_{\C_p}$ as $\C_p\cup\{\infty\}$, we introduce the spherical metric on  $\mathbb{P}^{1}_{\C_p}$ as follows: for $x,y \in \C_p$,

 $$\rho(x,y)=\frac{|x-y|_{p}}{\max\{|x|_{p},1\}\max\{|y|_{p},1\}},$$
 and
 $$
 \rho(x,\infty)=\left\{
                    \begin{array}{ll}
                      1, & \mbox{if $|x|_{p}\leq 1$;} \\
                      1/|x|_{p}, & \mbox{if $|x|_{p}> 1$.}
                    \end{array}
                  \right.
 $$
 Note that for any $x,y\in\mathbb{P}^1_{\C_p}$, the distance $\rho(x,y)\le 1$.
 Since $\P^1_L\subset\P^1_{\C_p}$, for $0<r\le 1$, set the disks in metric $\rho$ as follows:
$$D_{\mathbb{P}^1_L}(x,r):=\{y\in\mathbb{P}^1_L:\rho(x,y)<r\} \ \text{and} \ \overline{D}_{\mathbb{P}^1_L}(x,r):=\{y\in\mathbb{P}^1_L:\rho(x,y)\le r\}.$$
Then $D_{\mathbb{P}^1_L}(x,r)$ and $\overline{D}_{\mathbb{P}^1_L}(x,r)$ are $\mathbb{P}^1_L$-disks. For a disk $D\subset\mathbb{P}^1_L$, denote its diameter by
$$\mathrm{diam}(D):=\sup_{x_1,x_2\in D}\rho(x_1,x_2).$$
For more details about the spherical metric, we refer to \cite[Section 5.1]{Benedetto19}.

\subsection{Analytic functions}\label{sec:analytic}
To simplify notations, in this subsection we write $|\cdot|$ for $|\cdot|_p$, and write $D(x,r)$ (resp. $\overline{D}(x,r)$) for  $D_{\mathbb{C}_p}(x,r)$ (resp. $\overline{D}_{\mathbb{C}_p}(x,r)$).

Let $D\subset\mathbb{C}_p$ be a disk and pick $x_0\in D$, a map $f: D\to L$ is {\it analytic} if it can be written as a power series
\begin{align*}\label{taylor-equ}
f(z)=\sum_{n=0}^\infty a_n(z-x_0)^n\in\C_p[[z-x_0]]
\end{align*}
converging on $D$.
 A point $c\in D$ is {\it critical} if the derivative of $f$ at $c$ equals $0$, i.e., $f'(c)=0$.
 If $D$ contains a critical point $c$, we can rewrite
 $$f(z)=b_0+\sum_{n=m}^\infty b_n(z-c)^n\in\C_p[[z-c]]$$
 with $b_m\not=0$ for some $m\ge 2$. We call $m$ the {\it local degree} of $f$ at $c$ and denote it by $\deg_cf$.

The following lemma shows that analytic maps are well behaved in a neighborhood of a non-critical point.
\begin{lemma}[\!\!{\cite[Chapter 3 Lemma 1.6]{Khrennikov04}}]\label{lem:local scaling poly}
Let $D\subset \mathbb{C}_p$ be a disk and let $f:D\to\mathbb{C}_p$ be an analytic map. Suppose $x_0\in D$ is not a critical point of $f$. Then there exists $r>0$ such that for all $2\le k<\infty$,
$$\left|\frac{f^{(k)}(x_0)}{k!}\right| r^{k-1}<|f'(x_0)|$$
and for all $x,y\in D(x_0,r)$,
$$|f(x)-f(y)|=|f'(x_0)||x-y|.$$
\end{lemma}
Inspired by the above result and following\cite[Definition 4.1]{Kingsbery09}, we define the following terminology of ``scaling" for convenience.
\begin{definition}\label{def:scaling}
An analytic map $f$ defined on a disk $D\subset\mathbb{C}_p$ is  \textit{(locally) scaling (in distances)} on $D(x_0,r)\subset D$ if there exists $\alpha\in|\mathbb{C}_p^\times|$ such that for all $x,y\in D(x_0,r)$,
$$|f(x)-f(y)|=\alpha |x-y|.$$
\end{definition}
We mention here that $f$ is scaling on $D(x_0,r)$ if and only if $f$ has Weierstrass degree $1$ in the common sense.
In the above definition, we call $\alpha>0$ the {\it scaling ratio} of $f$ at $x_0$. Moreover, we say the largest disk $D(x_0,r)$ on which $f$ is scaling is the \textit{maximal scaling disk} of $f$ at $x_0$, and denote it by $\mathcal{S}_{x_0}(f)$.  The following lemma asserts that $\mathcal{S}_{x_0}(f)$ is the maximal scaling disk of $f$ at any point $x\in\mathcal{S}_{x_0}(f)$.

\begin{lemma}\label{cor:same domain}
Under the assumptions in Lemma \ref{lem:local scaling poly}, if $x\in\mathcal{S}_{x_0}(f)$, then $\mathcal{S}_x(f)=\mathcal{S}_{x_0}(f)$.
\end{lemma}
\begin{proof}
It is obvious that $\mathcal{S}_{x_0}(f)\subset\mathcal{S}_x(f)$. By symmetry, $\mathcal{S}_{x_0}(f)\supset\mathcal{S}_x(f)$.
\end{proof}

At a critical point, we have the following property on distance of images.
\begin{lemma}\label{cor:local critical}
Let $D\subset\mathbb{C}_p$ be a disk and let $f:D\to\mathbb{C}_p$ be an analytic map. Assume  $c\in D$ is a critical point with $\deg_{c}f=m\ge 2$. Then there exists $r>0$ such that for all $x\in D(c,r)\subset D$,
$$|f(x)-f(c)|=\left|\frac{f^{(m)}(c)}{m!}\right||x-c|^{m}.$$
\end{lemma}
\begin{proof}
Set
$$g(x):=\frac{f(x)-f(c)}{(x-c)^{m-1}}.$$
Then $g(x)$ is analytic on $D$ and $g(c)=0$. By Lemma \ref{lem:local scaling poly}, there exists $r>0$ such that for all $x\in D(c,r)$,
\[
|g(x)-g(c)|=|g'(c)| |x-c|.
\]
Noting that $g'(c)=f^{(m)}(c)/m!$, we have
$$\left|\frac{f(x)-f(c)}{(x-c)^{m-1}}\right|=\left|\frac{f^{(m)}(c)}{m!}\right| |x-c|.$$
Hence the conclusion follows.
\end{proof}

Recall that a critical point $c$ of $f$ is called {\it tame} if $p\nmid \deg_{c}f$ and called {\it wild} if $p\mid \deg_{c}f$. Lemma \ref{cor:local critical} asserts that $f$ is not scaling in a neighborhood of any critical point. However, the next result claims that in a small neighborhood of any tame critical point $c$, the map $f$ is scaling on the disks not containing $c$. This property fails for wild critical points, in which case, $f$ is scaling on further small disks. This distinction is an obstruction that forces us to assume the absence of wild recurrent Julia critical points in our Theorem \ref{Thm:main}.

\begin{proposition}\label{prop:tame-wild}
Let $D\subset \mathbb{C}_p$ be a disk and let $f:D\to\mathbb{C}_p$ be an analytic map. Then the following hold:
\begin{enumerate}
\item Let $c\in D$ be a critical point of $f$. Then there exists  $r=r_c>0$ such that for any $x\in  D(c,r)$ with $x\not=c$, we have $\mathcal{S}_{x}(f)=D(x,\delta_f|x-c|)$, where
$$\delta_f=
\begin{cases}
1& \text{if $c$ is tame,}\\
p^{-\frac{1}{p-1}}<1 & \text{if $c$ is wild.}
\end{cases}$$

\item If in addition, $f$ has finitely many critical points in $D$, then for any finite extension $K$ of $\mathbb{Q}_p$ and for any small $r>0$, there exists $\epsilon_r>0$ such that $f$ is scaling on $D(x,\epsilon_r)\subset D$ for any $x\in(D\cap K)\setminus\bigcup\limits_{f'(c)=0}D(c,r)$.
\end{enumerate}
\end{proposition}
\begin{proof}
The statement (1) follows from \cite[Lemma 11.5]{Benedetto19}.
The statement (2) follows immediately from the local compactness of $K$ and the fact that all the critical points in $D\setminus K$ are uniformly away from the points in $K$.
\end{proof}

Note that for any point $x\in\mathbb{P}^{1}_{\mathbb{C}_p}$, in local coordinates, a rational map $\phi\in\C_p(z)$ is analytic in a neighborhood of $x$. Then Proposition \ref{prop:tame-wild} immediately implies the following result.

\begin{corollary}\label{coro:tame-wild}
Let $\phi\in\mathbb{C}_p(z)$ be a rational map of degree at least $2$. Then the following hold:
\begin{enumerate}
\item Let $c\in\mathbb{P}^1_{\mathbb{C}_p}$ be a critical point of $\phi$. Then there exists  $0<r=r_c<1$ such that for any $x\in  D_{\mathbb{P}^1_{\mathbb{C}_p}}(c,r)$ with $x\not=c$, in local coordinates, the map $\phi$ is scaling on $D_{\mathbb{P}^1_{\mathbb{C}_p}}(x,\delta_\phi\rho(x,c))$, where $\delta_\phi=1$ if $c$ is tame and $\delta_\phi=p^{-1/(p-1)}$ if $c$ is wild.
\item For any small $0<r<1$, there exists $0<\epsilon_r<1$ such that, in local coordinates, $\phi$ is scaling on $D_{\mathbb{P}^1_{\mathbb{C}_p}}(x,\epsilon_r)$ for any $x\in\mathbb{P}^1_ K\setminus\bigcup\limits_{c\in\mathrm{Crit}_{\mathbb{C}_p}(\phi)}D_{\mathbb{P}^1_{\mathbb{C}_p}}(c,r)$.
\end{enumerate}
\end{corollary}

For the images of disks near a wild critical point, we have the following estimates for the ratios of diameters of disks, which will play a crucial role in the proof of Lemma \ref{lem:uniform-scaling}.
\begin{corollary}\label{coro:image-ratio}
Let $\phi\in\mathbb{C}_p(z)$ be a rational function of degree at least $2$. Suppose that $c\in\mathrm{Crit}_{\mathbb{C}_p}(\phi)$ is a wild critical point and let $r_c>0$ be as in Corollary \ref{coro:tame-wild}. For $0<r<r_c$ sufficiently small, consider disks $D_1\subset D_2\subset D_{\mathbb{P}^1_{\mathbb{C}_p}}(c,r)$. Assume $c\not\in D_2$. Then
$$\frac{\mathrm{diam}(\phi(D_1))}{\mathrm{diam}(\phi(D_2))}
\begin{cases}
\medskip
\ge p^{-1/(p-1)}\cdot |\deg_c\phi|&\text{if}\ \phi \ \text{is not scaling on}\ D_1,\\

\medskip
\ge\frac{\mathrm{diam}(D_1)}{\rho(c,D_2)}|\deg_c\phi|&\text{if}\ \phi\ \text{is not scaling on}\ D_2\ \text{but scaling on}\ D_1,\\
=\frac{\mathrm{diam}(D_1)}{\mathrm{diam}(D_2)}&\text{if}\ \phi \ \text{is scaling on}\ D_2.\\
\end{cases}$$
Moreover, if in addition, $c\in\mathbb{P}^1_K$ and $D_1\cap\mathbb{P}^1_K\not=\emptyset$, then for any $\beta\in |K^\times|$ with $\beta\le p^{-\frac{1}{p-1}}$,
$$\frac{\mathrm{diam}(\phi(D_1\cap\mathbb{P}^1_K))}{\mathrm{diam}(\phi(D_2))}\
\begin{cases}
\medskip
\ge \beta|\deg_c\phi|&\text{if}\ \phi \ \text{is not scaling on}\ D_1,\\
\medskip
\ge\frac{\mathrm{diam}(D_1)}{\rho(c,D_2)}|\deg_c\phi|&\text{if}\ \phi\ \text{is not scaling on}\ D_2\ \text{but scaling on}\ D_1,\\
=\frac{\mathrm{diam}(D_1)}{\mathrm{diam}(D_2)}&\text{if}\ \phi \ \text{is scaling on}\ D_2.\\
\end{cases}$$
\end{corollary}
\begin{proof}
We first consider the case that $c=0$.
Pick $x\in D_1$ and consider the maximal scaling disk $\mathcal{S}_x(\phi)=D_{\mathbb{P}^1_{\mathbb{C}_p}}(x,p^{-1/(p-1)}\rho(x,c))$. It follows that $\phi(\mathcal{S}_x(\phi)\cap D_1)\subset \phi(D_1)$, and hence
$$\mathrm{diam}(\phi(\mathcal{S}_x(\phi)\cap D_1))\le\mathrm{diam}(\phi(D_1)).$$
Since $r>0$ is sufficiently small,
$$\mathrm{diam}(\phi(\mathcal{S}_x(\phi)\cap D_1))=|\phi'(x)|\mathrm{diam}(\mathcal{S}_x(\phi)\cap D_1).$$
Moreover, letting $a\in\mathbb{C}_p$ be the coefficient of the term $(x-c)^{\deg_c\phi}$ in the Taylor series of $\phi$ near $c$, we have
$$|\phi'(x)|=|a\deg_c\phi| |x-c|^{\deg_c\phi-1}=|a\deg_c\phi|\rho(x,c)^{\deg_c\phi-1},$$
 and
$$\mathrm{diam}(\phi(D_2))
\begin{cases}
\medskip
\le |a|\rho(x,c)^{\deg_c\phi}&\text{if}\ \phi \ \text{is not scaling on}\ D_2,\\
= |\phi'(x)|\mathrm{diam}(D_2)&\text{if}\ \phi \ \text{is scaling on}\ D_2.
\end{cases}$$
Then we conclude that
\begin{align*}
\medskip
\frac{\mathrm{diam}(\phi(D_1))}{\mathrm{diam}(\phi(D_2))}&\ge\frac{|\phi'(x)|\mathrm{diam}(\mathcal{S}_x(\phi)\cap D_1)}{\mathrm{diam}(\phi(D_2))}\\
&\begin{cases}
\medskip
\ge \frac{|\deg_c\phi|\mathrm{diam}(\mathcal{S}_x(\phi)\cap D_1)}{\rho(x,c)}&\text{if}\ \phi \ \text{is not scaling on}\ D_2,\\
= \frac{\mathrm{diam}(D_1)}{\mathrm{diam}(D_2)}&\text{if}\ \phi \ \text{is scaling on}\ D_2.
\end{cases}
\end{align*}
Note that $\mathrm{diam}(\mathcal{S}_x(\phi))=p^{-\frac{1}{p-1}}\rho(x,c)$. Then the first assertion follows immediately for $c=0$.

In the case that $D_1\cap\mathbb{P}^1_K\not=\emptyset$, we pick $x\in D_1\cap\mathbb{P}^1_K$. If $\phi$ is scaling on $D_1$ in local coordinates, then $\mathrm{diam}(\phi(D_1\cap\mathbb{P}^1_K))=\mathrm{diam}(\phi(D_1))$, which implies that, in this case, the second assertion holds for $c=0$. Now we consider the case that $\phi$ is not scaling on $D_1$. Then for any $\beta\in |K^\times|$ with $\beta\le p^{-\frac{1}{p-1}}$, we have
$$D_{\mathbb{P}^1_{\mathbb{C}_p}}(x,\beta\rho(x,c))\subset\mathcal{S}_x(\phi)\subset D_1.$$
Note that $\mathrm{diam}(\mathcal{S}_x\cap\P^1_K)\le\mathrm{diam}(\mathcal{S}_x)$ and
$$\mathrm{diam}(D_{\mathbb{P}^1_{\mathbb{C}_p}}(x,\beta\rho(x,c))\cap\mathbb{P}^1_K)=\mathrm{diam}(D_{\mathbb{P}^1_{\mathbb{C}_p}}(x,\beta\rho(x,c))).$$
It follows that
$$\mathrm{diam}(\phi(D_1)\cap\mathbb{P}^1_K)\ge \mathrm{diam}(\phi(\mathcal{S}_x\cap\mathbb{P}^1_K))\ge\mathrm{diam}(\phi(D_{\mathbb{P}^1_{\mathbb{C}_p}}(x,\beta\rho(x,c))))=|\phi'(x)|\beta\rho(x,c).$$
Hence
$$\frac{\mathrm{diam}(\phi(D_1)\cap\mathbb{P}^1_K)}{\mathrm{diam}(\phi(D_2))}\ge\frac{\mathrm{diam}(\phi(\mathcal{S}_x\cap\mathbb{P}^1_K))}{\mathrm{diam}(\phi(D_2))}\ge\frac{|\phi'(x)|\beta\rho(x,c)}{|a|\rho(x,c)^{\deg_c\phi}}.$$
Thus in this case, the second assertion also holds for $c=0$.

In general, set $\psi(z)=z+c$ and consider the map $\phi\circ\psi$. Let $\widetilde{D}_1=\{x-c: x\in D_1\}$ and  $\widetilde{D}_2=\{x-c: x\in D_2\}$. Note that
$$\frac{\mathrm{diam}(\phi(D_1))}{\mathrm{diam}(\phi(D_2))}=\frac{\mathrm{diam}(\phi\circ\psi(\widetilde{D}_1))}{\mathrm{diam}(\phi\circ\psi(\widetilde{D}_2))}.$$
Applying the previous argument for  $\phi\circ\psi$, we obtain the conclusion.
\end{proof}

\subsection{Julia sets and Fatou sets}\label{sec:Julia-Fatou}
Let $\phi\in L(z)$ be a rational map of degree at least $2$. The \textit{Fatou set} $F_L(\phi)$ is defined to be the subset of points in $\mathbb{P}^1_L$ having a neighborhood on which the family $\{\phi^n\}_{n\ge 0}$  of iterations is equicontinuous with respect to the spherical metric $\rho$. The \textit{Julia set} $J_L(\phi)$ is the complement $\mathbb{P}^1_L\setminus F_L(\phi)$. Hence the set $F_L(\phi)$ is open while the set $J_L(\phi)$ is closed. If $L$ is algebraically closed (i.e. $L=\C_p$), the set $F_L(\phi)$ contains all the nonrepelling periodic points, while the set $J_L(\phi)$ contains all the repelling periodic points, see \cite[Proposition 1.1]{Benedetto01}. Moreover, the map $\phi$ always has a nonrepelling fixed point and hence $F_L(\phi)\not=\emptyset$, see \cite[Corollary 5.19]{Silverman07}. In contrast, $J_L(\phi)$ could be empty. For example, if $\phi$ has good reduction, then $J_L(\phi)=\emptyset$, see \cite[Theorem 2. 17]{Silverman07}. In this case, both sets $F_L(\phi)$ and $J_L(\phi)$ are totally invariant under $\phi$ \cite[Proposition 5.10]{Benedetto19}.
However, if $L$ is not algebraically closed (i.e. $L=K$), the set $F_L(\phi)$ could be empty, see \cite{Fan20}. In this case, even though the set  $J_L(\phi)$ is forward invariant and contains all the repelling periodic points in $\P^1_L$, the total invariance of $J_L(\phi)$ is unclear. However, we have the following invariance of $J_K(\phi)$ away from critical points.

\begin{lemma}\label{regular-point-lemma}
Let $\phi\in K(z)$ be a rational map of degree at least $2$.  Assume that $x\in\mathbb{P}^{1}_{K}$ is not a critical point of $\phi$.  Then $x\in J_K(\phi)$ if and only if $\phi(x)\in J_K(\phi)$.
\end{lemma}
\begin{proof}
Instead of $J_K(\phi)$, let us work on $F_K(\phi)$. By the definition of $F_K(\phi)$, we immediately have that if $x\in F_K(\phi)$, then $\phi(x)\in F_K(\phi)$. Now we assume $\phi(x)\in F_K(\phi)$. Then there exists a neighborhood $V\subset\mathbb{P}^1_K$ of $\phi(x)$ such that the sequence  $\{\phi^n\}_{n\ge 1}$ is equicontinuous on $V$. Since $x$ is not a critical point, by Lemma \ref{lem:local scaling poly}, there exists a neighborhood $U\subset\mathbb{P}^1_K$ of $x$ such that $\phi$ is  scaling on $U$ and $\phi(U)\subset V$. Hence $\{\phi^{n+1}\}_{n\ge 1}$ is equicontinuous on $U$. Thus $\{\phi^{n}\}_{n\ge 1}$ is equicontinuous on $U$. We conclude $x\in F_K(\phi)$.
\end{proof}

The following result is an immediate consequence of Lemma \ref{regular-point-lemma} that we will use repeatedly in our later argument.
\begin{corollary}\label{coro:scaling-julia}
Let $\phi\in K(z)$ be a rational map of degree at least $2$, and let $D\subset\mathbb{P}^1_{\mathbb{C}_p}$ be a disk such that $\phi^n$ is scaling on $D$ for some $n\ge 1$. Suppose $\phi^n(D\cap\mathbb{P}^1_K)\cap J_K(\phi)\not=\emptyset$. Then $(D\cap\mathbb{P}^1_K)\cap J_K(\phi)\not=\emptyset$.
\end{corollary}

For a rational map $\phi\in\mathbb{C}_p(z)$, the following result asserts that any disk intersecting the Julia set has an iteration with large diameter.

\begin{lemma}\label{lem:dimater-large}
Let $\phi\in \mathbb{C}_p(z)$ be a rational map of degree at least $2$, and let $D\subset\mathbb{P}^1_{\mathbb{C}_p}$ be a disk such that $D\cap J_{\mathbb{C}_p}(\phi)\not=\emptyset$. Then there exists $n_0\ge 0$ such that
$$\mathrm{diam}(\phi^{n_0}(D))=1.$$
\end{lemma}
\begin{proof}
Suppose to the contrary that $\mathrm{diam}(\phi^n(D))<1$ for all $n\ge 0$. Let $\theta_n\in\mathrm{PGL}(2,\mathbb{C}_p)$ be an isometry such that $\theta_n(\phi^n(D))\subset D_{\mathbb{C}_p}(0,1)$. It follows from \cite[Corollary 5.18]{Benedetto19} that the sequence $\{\theta_n\circ \phi^n\}_{n\ge 0}$ is equicontinuous on $D$ with respect to the spherical metric.  Applying \cite[Lemma 5.6]{Benedetto19}, we obtain that $\{\phi^n\}_{n\ge 0}$ is equicontinuous on $D$, and hence $D\subset F_{\mathbb{C}_p}(\phi)$, which contradicts the assumption $D\cap J_{\mathbb{C}_p}(\phi)\not=\emptyset$.
\end{proof}

\subsection{Fatou components}\label{sec:Fatou-components} For a rational map $\phi\in\mathbb{C}_p(z)$ of degree at least $2$, it is possible to define the Fatou components in terms of D-components or analytic components, see \cite{Benedetto02}. However, it is more convenient for our purposes to inherit the terminologies from the Berkovich dynamics.

Let $\mathbb{P}^{1,an}_{\mathbb{C}_p}$ be the Berkovich space over $\mathbb{C}_p$ which is a compact, uniquely path connected, Hausdroff topological space, see \cite{BRbook2010, Berkovich90, RL2003Ast, Rivera03a, Rivera05a}. Since $\mathbb{P}^1_{\mathbb{C}_p}$ can be identified with a dense subset in $\mathbb{P}^{1,an}_{\mathbb{C}_p}$, the rational map $\phi$ extends to a self map on $\mathbb{P}^{1,an}_{\mathbb{C}_p}$, see \cite{BRbook2010}. To abuse notation, we also call this extension $\phi$. The {\it Berkovich Julia set} $J_{Ber}(\phi)$ consists of the points $\xi\in\mathbb{P}^{1,an}_{\mathbb{C}_p}$ such that for all (weak) neighborhood $U$ of $\xi$, the set $\cup_{n=1}^\infty\phi^n(U)$ omits at most two points in $\mathbb{P}^{1,an}_{\mathbb{C}_p}$. The complement of $J_{Ber}(\phi)$ is the {\it Berkovich Fatou set} $F_{Ber}(\phi)$. For the equivalent definitions for  $J_{Ber}(\phi)$ and  $F_{Ber}(\phi)$ in terms of appropriate notions of equicontinuity or normality, we refer to \cite{Favre12}.

Following \cite[Theorem 8.3]{Benedetto19}, we have that $F_{\mathbb{C}_p}(\phi)$ is a subset of $F_{Ber}(\phi)$, more precisely, $F_{Ber}(\phi)\cap\mathbb{P}^{1}_{\mathbb{C}_p}=F_{\mathbb{C}_p}(\phi)$.
Then we say a subset $\Omega\subset F_{\mathbb{C}_p}(\phi)$ is a {\it (Fatou) component} of $F_{\mathbb{C}_p}(\phi)$ if there exists a component $U$ of $F_{Ber}(\phi)$ such that $\Omega=U\cap F_{\mathbb{C}_p}(\phi)$. In our term, a component of $F_{\mathbb{C}_p}(\phi)$ is a union of possibly infinitely many D-components in \cite{Benedetto02}.

The following result is an immediate consequence of the fact that each component of $F_{Ber}(\phi)$ is mapped onto a component of $F_{Ber}(\phi)$ by $\phi$, see \cite[Proposition 8.2]{Benedetto19}.

\begin{proposition}
Let $\phi \in  \C_p(z)$ be a rational function of degree at least $2$, and let $\Omega$ be a component of $F_{\mathbb{C}_p}(\phi)$. Pick $x\in\Omega$. Then $\phi(\Omega)$ is the component of $F_{\mathbb{C}_p}(\phi)$ containing $\phi(x)$. Moreover, $\phi^{-1}(\Omega)$ is the union of components of  $F_{\mathbb{C}_p}(\phi)$ containing the preimages $\phi^{-1}(x)$.
\end{proposition}

A periodic component of $F_{Ber}(\phi)$ is either attracting which contains an attracting periodic point, or indifferent which is bijective to itself under certain iteration, see \cite[Section 4.4]{RL2003Ast}. The components of $F_{\C_p}(\phi)$ inherit immediately classification of periodic components of $F_{Ber}(\phi)$.

\begin{corollary}\label{Classification-theorem}
Let $\phi \in  \C_p(z)$ be a rational function of degree at least $2$, and let $\Omega$ be a periodic connected component of $F_{\mathbb{C}_p}(\phi)$. Then $\Omega$ is either an attracting component containing an attracting periodic point, or an indifferent component mapping bijectively to itself under certain iteration.
\end{corollary}

A wandering domain is a component of $F_{\C_p}(\phi)$ that is not eventually periodic. In contrast to the complex setting (see \cite{Sullivan85}), there exist polynomials in $\mathbb{C}_p[z]$ possessing wandering domains, see \cite[Theorem 1.1]{Benedetto02I}.
However, for rational maps of degree at least $2$ with algebraic coefficients, it is conjectured that such maps have no wandering domains in $\mathbb{P}^{1}_{\C_p}$, see \cite[Conjecture 5.50]{Silverman07}.
This conjecture holds under the no  wild recurrent Julia critical points assumption, see \cite[Theorem 1.2]{Benedetto00}. For our purposes, we state the following no wandering domains result that is slightly stronger than Benedetto's original version. The proof is a minor modification of Benedetto's proof in \cite[Theorems 1.2 and 1.3]{Benedetto00}, see also \cite[Section 11.3]{Benedetto19}. We omit the proof here.

\begin{theorem}\label{no-Wandering-thm}
Let $\phi\in K(z)$ be a rational map of degree at least $2$. Suppose $\phi$ has no wild recurrent  critical  points in $J_{\mathbb{C}_p}(\phi)\cap\mathbb{P}^1_K$. Then every component of $F_{\mathbb{C}_p}(\phi)$ that intersects with $\mathbb{P}^1_K$ is non-wandering. Moreover, there are only finitely many periodic components of $F_{\mathbb{C}_p}(\phi)$ intersecting $\mathbb{P}^1_K$.  
\end{theorem}

\subsection{Brief background of symbolic dynamics}\label{sec:symbols}
We now state some preliminaries of symbolic dynamics for later use. Standard references are \cite{Kitchens98} and \cite{Lind95}.

Let $\mathcal{A}$ be a countable symbol set. Associated with the discrete topology, the set $\mathcal{A}$ is noncompact.  Then the product space $\Sigma:=\mathcal{A}^{\mathbb{N}}$ is also noncompact under the product topology.
 For $b_0,b_1,\cdots,b_j\in\mathcal{A}$, the \textit{cylinder set} starting $b_0,b_1,\cdots,b_j$ is
 $$[b_0b_1\cdots b_j]:=\{a=(a_i)_{i\in\mathbb{N}}\in\Sigma: a_i=b_i\ \text{for}\ 0\le i\le j\}.$$
 All the cylinder sets form a countable basis of open-closed sets. Moreover, we can define a distance between two distinct points $a=\{a_i\}_{i\in\mathbb{N}}$ and $a'=\{a'_i\}_{i\in\mathbb{N}}$ in $\Sigma$ by $\tilde\rho(a,a')=2^{-n}$, where $n$ is the smallest integer with $a_n\not=a'_n$. Then the metric $\tilde\rho$ induces the same topology as the product topology.

Let $\sigma:\Sigma\to\Sigma$ be the (left-)shift. A countable $|\mathcal{A}|\times|\mathcal{A}|$, zero-one matrix $A=(A_{b_ib_j})$ induces a $\sigma$-invariant subset $\Sigma_A$ of $\Sigma$:
$$\Sigma_A=\{(b'_i)\in\Sigma: A_{b'_ib'_{i+1}}=1\},$$
which is noncompact under the subspace topology. We say the system $(\Sigma_A,\sigma)$ is a \textit{countable state Markov shift} defined by $A$, and call $(\Sigma_A,\sigma)$ \textit{irreducible} if the matrix $A$ is irreducible, i.e., for any $b_i, b_j\in \mathcal{A}$, there exists a positive integer $\ell$ such that $(A^\ell)_{b_ib_j}>0$.

Since the space $\Sigma_A$ is noncompact, we consider the Gurevich entropy $h_G(\Sigma_A,\sigma)$ for the system $(\Sigma_A,\sigma)$, see \cite{Gurevivc69} and \cite{Gurevivc70}, which is the supremum of the topological entropies $ h_{top}(\Sigma'_A, \sigma)$ over all subsystems $(\Sigma'_A,\sigma)$ formed by restricting to a finite subset of symbols in $\mathcal{A}$. If the matrix $A$ is irreducible, there is a combinatorial way to compute $h_G(\Sigma_{A}, \sigma)$: For any $a_{i_0}\in\mathcal{A}$, a {\it first-return loop} of length $n\ge 1$ at $a_{i_0}$ is a path $\{a_{i_0},a_{i_1}, \cdots, a_{i_n}\}$ such that $a_{i_0}=a_{i_n}$, $a_{i_k}\not=a_{i_0}$ for $1\le k\le n-1$ and $A_{a_{i_k},a_{i_{k+1}}}=1$ for $0\le k\le n-1$. Let $\Xi_{a_{i_0}}(n)$ be the number of the first-return loops at $a_{i_0}$ of length $n$, and set
$$G_{a_{i_0}}(z)=\sum_{n\ge 1}\Xi_{a_{i_0}}(n) z^n.$$
If $1-G_{a_{i_0}}(z)$ has a real root $R>0$ such that $G_{a_{i_0}}(z)$ converges and is not $1$ on $|z|<R$, then $h_G(\Sigma_{A}, \sigma)=-\log R$. Remark that if such an $R$ exists, then it is independent of $a_{i_0}$, see \cite{Ruette03, Vere-Jones62}.

\section{Restriction of the Julia set}\label{restriction}
In this section, the goal is to prove Theorem \ref{Thm:main}.
Let $\phi\in K(z)$ be a rational map of degree at least $2$. As mentioned in Section \ref{sec:intro}, by definitions, $J_K(\phi)\subset J_{\mathbb{C}_p}^{K}(\phi)$. We will show $J_{\mathbb{C}_p}^{K}(\phi)\subset J_K(\phi)$. Before that, we mention here the assumption (2) of Theorem \ref{Thm:main} is not required until Section \ref{sec:recurrent}. But we repeatedly apply the assumption (1), that is, in this section, we always make the following assumption.

\smallskip

 \textbf{Assumption:} In this section, the map $\phi$ has no wild recurrent critical points in $J^K_{\mathbb{C}_p}(\phi)$.

 \smallskip

Let $\partial J_{\mathbb{C}_p}^{K}(\phi)$ be the boundary of $J_{\mathbb{C}_p}^{K}(\phi)$ in $\mathbb{P}^{1}_{K}$, and denote by
$$\mathrm{Int}(J_{\mathbb{C}_p}^{K}(\phi)):=J_{\mathbb{C}_p}^{K}(\phi)\setminus \partial J_{\mathbb{C}_p}^{K}(\phi)$$
the interior of $J_{\mathbb{C}_p}^{K}(\phi)$ in $\mathbb{P}^{1}_{K}$.

\subsection{Critical distances}\label{sec:distance}
We state the following quantities on the distances concerning critical orbits for later use.  For each $c\in\mathrm{Crit}_{\mathbb{C}_p}(\phi)$, fix an $r_{c}>0$ satisfying Corollary \ref{coro:tame-wild} (1). We set
\begin{align*}
r_1&:=\min\left\{r_{c}:c\in\mathrm{Crit}_{\mathbb{C}_p}(\phi)\right\},\\
r_2&:=\min\left\{\rho\left(c,\ \mathrm{Crit}_{\C_p}(\phi)\setminus\{c\}\right):c\in\mathrm{Crit}_{\C_p}(\phi)\right\},\\
r_3&:=\min\left\{\rho\left(\overline{\mathcal{O}_\phi(c)},\ \mathrm{Crit}_{\C_p}(\phi)\setminus\overline{\mathcal{O}_\phi(c)}\right): c\in\mathrm{Crit}_{\C_p}(\phi)\right\},\\
r_4&:=\min\left\{\rho\left(c,J_{\mathbb{C}_p}(\phi)\right): c\in \mathrm{Crit}_{\mathbb{C}_p}(\phi)\cap F_{\mathbb{C}_p}(\phi)\right\},\ \  \text{and}\ \\
r_5&:=\min\left\{\rho\left(c,\mathbb{P}^1_K\right): c\in\mathrm{Crit}_{\mathbb{C}_p}(\phi)\setminus\mathrm{Crit}_K(\phi)\right\}.
\end{align*}
Due to the finiteness of $\mathrm{Crit}_{\C_p}(\phi)$,  the closedness of the closure $\overline{\mathcal{O}_\phi(c)}$ and $J_{\mathbb{C}_p}(\phi)$, and the compactness of $\P^1_K$, each of the above quantities is strictly positive. Thus
$$r_\ast:=\min\{r_1,r_2,r_3,r_4,r_5\}>0.$$

\subsection{Points in $\partial J_{\mathbb{C}_p}^{K}(\phi)$} \label{sec:boundary}

The following result, whose proof combines the classification result (Corollary \ref{Classification-theorem}) and the no wandering domains result (Theorem \ref{no-Wandering-thm}), asserts that $\partial J_{\mathbb{C}_p}^{K}(\phi)\subset J_K(\phi)$.

\begin{lemma}\label{boundary-julia-lem}
If $x\in \partial J_{\mathbb{C}_p}^{K}(\phi)$, then  $x\in J_K(\phi)$.
\end{lemma}
\begin{proof}
Since $x\in \partial J_{\mathbb{C}_p}^{K}(\phi)$, for any small neighborhood  $V\subset\mathbb{P}^{1}_{K}$ of $x$, the set  $V\cap F_{\mathbb{C}_p}^{K}(\phi)$ is non-empty.  By Theorem \ref{no-Wandering-thm},  for each  $y\in V\cap F_{\mathbb{C}_p}^{K}(\phi)$, there exists an integer $m\ge 0$ such that $\phi^{m}(y)\in\mathbb{P}^1_K$ lies in a periodic component of $F_{\C_p}(\phi)$.  Again by Theorem \ref{no-Wandering-thm}, the map $\phi$ has finitely many periodic components of $F_{\C_p}(\phi)$ intersecting $\mathbb{P}^1_K$, and by Corollary \ref{Classification-theorem}, the periodic components of $F_{\C_p}(\phi)$ are either attracting or indifferent. Then there exists a constant $\epsilon>0$ such that  for all $y \in  V\cap F_{\mathbb{C}_p}^{K}(\phi)$,
\[\sup_{i\geq 0} \rho(\phi^i(x),\phi^{i}(y))>\epsilon.\]
This means that $x\in J_K(\phi)$.
\end{proof}

\subsection{Reduction to critical points}\label{sec:reduce}

In this subsection, we show that in order to establish $J_{\mathbb{C}_p}^{K}(\phi)\subset J_K(\phi)$, it only requires to treat the critical points in $J^K_{\mathbb{C}_p}(\phi)$. We first state the following characterization of the orbit of a disk intersecting $J^K_{\mathbb{C}_p}(\phi)$.
\begin{lemma}\label{lem:uniform-scaling}
Let $D\subset \mathbb{P}^1_K$ be a disk such that $D\cap J^K_{\mathbb{C}_p}(\phi)\not=\emptyset$. Then at least one of the following holds:
\begin{enumerate}
\item There exists $n_0\ge 0$ such that
$$\phi^{n_0}(D)\cap\mathrm{Crit}_K(\phi)\not=\emptyset.$$
\item There exists $\epsilon_\ast>0$, independent of $D$, such that
$$\sup_{n\ge 1}\mathrm{diam}(\phi^n(D))\ge\epsilon_\ast.$$
\end{enumerate}
\end{lemma}
\begin{proof}
Pick $0<r<r_\ast$ and let $\epsilon_r>0$ be as in Corollary \ref{coro:tame-wild} (2). Shrinking $\epsilon_r$ if necessary, we can assume that $0<\epsilon_r<r$.
Consider the disk $D'\subset\mathbb{P}^1_{\mathbb{C}_p}$ such that $D\subset D'$ and $\mathrm{diam}(D')=\mathrm{diam}(D)$.
Then the assumption $D\cap J^K_{\mathbb{C}_p}(\phi)\not=\emptyset$ implies $D'\cap J_{\mathbb{C}_p}(\phi)\not=\emptyset$.
Thus $\mathbb{P}^1_{\mathbb{C}_p}\setminus\bigcup_{n\ge 1}\phi^n(D')$ contains at most one point (see \cite[Theorem 5.19]{Benedetto19}). Since $\phi$ has at least two distinct critical points in $\mathrm{Crit}_{\mathbb{C}_p}(\phi)$, the set $\bigcup_{n\ge 1}\phi^n(D')$ intersects $\mathrm{Crit}_{\mathbb{C}_p}(\phi)$. Hence there exists a smallest $\ell\ge 0$ such that $\phi$ is not scaling on $\phi^{\ell}(D')$. It follows that
\begin{equation}\label{equ:same-11}
\mathrm{diam}(\phi^{\ell}(D))=\mathrm{diam}(\phi^{\ell}(D')).
\end{equation}

Now suppose that the statement (1) does not hold. In the following, we show that the statement (2) holds. By the choice of $\ell$, we have
$$\phi^{\ell}(D')\cap\mathrm{Crit}_K(\phi)=\emptyset.$$
If $\phi^{\ell}(D')$ contains a critical point $c\in\mathrm{Crit}_{\mathbb{C}_p}(\phi)$, then $c\in\mathrm{Crit}_{\mathbb{C}_p}(\phi)\setminus\mathrm{Crit}_K(\phi)$. Since $\phi^{\ell}(D')\cap\mathbb{P}^1_K\not=\emptyset$, by the choice of $r_\ast$ and \eqref{equ:same-11}, we conclude
\begin{equation}\label{equ:same-12}
\mathrm{diam}(\phi^{\ell}(D))=\mathrm{diam}(\phi^{\ell}(D'))>r_\ast.
\end{equation}
If $\phi^{\ell}(D')\cap D_{\mathbb{P}^1_{\mathbb{C}_p}}(c,r)=\emptyset$ for all $c\in \mathrm{Crit}_{\mathbb{C}_p}(\phi)$, since $\phi$ is not scaling on $\phi^{\ell}(D')$, we obtain
\begin{equation}\label{equ:same-13}
\mathrm{diam}(\phi^{\ell}(D))=\mathrm{diam}(\phi^{\ell}(D'))\ge\epsilon_r.
\end{equation}

To proceed the argument, assume $$\phi^{\ell}(D')\cap \mathrm{Crit}_{\mathbb{C}_p}(\phi)=\emptyset \quad \text{but} \quad \phi^{\ell}(D')\cap D_{\mathbb{P}^1_{\mathbb{C}_p}}(c',r)\not=\emptyset$$ for some $c'\in\mathrm{Crit}_{\mathbb{C}_p}(\phi)$. It follows that
\begin{equation}\label{equ:contain-11}
\phi^{\ell}(D')\subsetneq D_{\mathbb{P}^1_{\mathbb{C}_p}}(c',r).
\end{equation}
Moreover, we have the following claim on the critical point $c'$.
\begin{claim}\label{c:1}
The critical point $c'\in J^K_{\mathbb{C}_p}(\phi)$ is wild.
\end{claim}
We first show $c'\in J^K_{\mathbb{C}_p}(\phi)$. Suppose otherwise. Then
$$c'\in\left(\mathrm{Crit}_{\mathbb{C}_p}(\phi)\setminus\mathrm{Crit}_K(\phi)\right)\bigcup F^K_{\mathbb{C}_p}(\phi).$$
Since  $D\cap J^K_{\mathbb{C}_p}(\phi)\not=\emptyset$  and $\phi^\ell$ is scaling on $D'$, we obtain $\phi^{\ell}(D')\cap J^K_{\mathbb{C}_p}(\phi)\not=\emptyset$. Hence by the choice of $r_\ast$, $D_{\mathbb{P}^1_{\mathbb{C}_p}}(c',r)\cap J^K_{\mathbb{C}_p}(\phi)=\emptyset$, which contradicts \eqref{equ:contain-11}.
Now we show $c'$ is wild. Suppose to the contrary that $c'$ is tame. Then by Corollary \ref{coro:tame-wild} (1) and the choice of $\ell$, we obtain $c'\in\phi^{\ell}(D')$, which contradicts the assumption $\phi^{\ell}(D')\cap \mathrm{Crit}_{\mathbb{C}_p}(\phi)=\emptyset$. Thus the claim holds.

\smallskip
Since $\phi$ is not scaling on $\phi^\ell(D')$, by Corollary \ref{coro:image-ratio}, there exists $\beta\in|K^\times|$ with $0<\beta\le p^{-\frac{1}{p-1}}$ such that
$$\frac{\mathrm{diam}(\phi^{\ell+1}(D))}{\mathrm{diam}(\phi^{\ell+1}(D'))}\ge\beta|\deg_{c'}\phi|_p.$$
By Claim \ref{c:1}, let $\mathcal{W}_\phi(c')=\{c_0=c',c_1,\cdots,c_{k_0}\}$ be the set of wild critical points in $\overline{\mathcal{O}_\phi(c')}\subset J^K_{\mathbb{C}_p}(\phi)$. The no wild recurrent critical points assumption implies that each point in $\mathcal{W}_\phi(c')$ is non-recurrent. Then we have the following claim on the lower bound of the diameter of certain iterated images of $D'$ and $D$.
\begin{claim}\label{c:2}
There exists $\ell'\ge0$ such that $\mathrm{diam}(\phi^{\ell'}(D'))\ge\epsilon_r$ and
$$\mathrm{diam}(\phi^{\ell'}(D))\ge\beta\left(p^{-\frac{1}{p-1}}\right)^{k_0}\prod_{i=0}^{k_0}|\deg_{c_i}\phi|_p\epsilon_r.$$
\end{claim}
Let us prove Claim \ref{c:2}. Note that $D'\cap J_{\mathbb{C}_p}(\phi)\not=\emptyset$. By Lemma \ref{lem:dimater-large}, there exists a smallest $\ell'\ge 0$ such that
\begin{equation}\label{equ:bigger}
\mathrm{diam}(\phi^{\ell'}(D'))\ge\epsilon_r.
\end{equation}
 If $0\le\ell'\le\ell$, since $\phi^\ell$ is scaling on $D'$, by \eqref{equ:same-11}, we conclude
 $$\mathrm{diam}(\phi^{\ell}(D))=\mathrm{diam}(\phi^{\ell}(D'))\ge\epsilon_r>\beta\left(p^{-\frac{1}{p-1}}\right)^{k_0}\prod_{i=0}^{k_0}|\deg_{c_i}\phi|_p\epsilon_r.$$
 Thus in this case, the claim holds immediately. Now we work on the case that $\ell'>\ell$. Relabeling the points in $\mathcal{W}_\phi(c')$ with $c_0=c'$ and applying Corollary \ref{coro:tame-wild} (1), there exist $0\le j_0\le k_0$ and $\ell=\ell_0\le\ell_{j-1}<\ell_j\le\ell'$ for $1\le j\le j_0$ such that $\phi^{\ell_j}(D')\subsetneq D_{\mathbb{P}^1_{\mathbb{C}_p}}(c_j,r)$ and $\phi^{\ell_j-\ell_{j-1}-1}$ is scaling on $\phi^{\ell_{j-1}+1}(D')$. It follows that if $\ell_j+1\le i\le\ell_{j+1}$ for some $1\le j\le j_0-1$, or if $\ell_{j_0}+1\le i\le\ell'$, then
\begin{align}\label{equ:ratio-11}
\frac{\mathrm{diam}(\phi^{i}(D))}{\mathrm{diam}(\phi^i(D'))}=\frac{\mathrm{diam}(\phi^{\ell_j+1}(D))}{\mathrm{diam}(\phi^{\ell_j+1}(D'))}.
\end{align}
Moreover, letting $D'_j\subset\mathbb{P}^1_{\mathbb{C}_p}$ be the disk such that $\phi^{\ell_j}(D)\subset D'_j$ and $\mathrm{diam}(D'_j)=\mathrm{diam}(\phi^{\ell_j}(D))$, if $\phi$ is scaling on $D'_j $ but not on $\phi^{\ell_j}(D')$, by Corollary \ref{coro:image-ratio}, we have
\begin{align}\label{equ:ratio-111}
\frac{\mathrm{diam}(\phi^{\ell_j+1}(D))}{\mathrm{diam}(\phi^{\ell_j+1}(D'))}\ge\frac{\mathrm{diam}(\phi^{\ell_{j}}(D))}{\rho(c_j,\phi^{\ell_{j}}(D'))}|\deg_{c_j}\phi|_p.
\end{align}

Let $0\le j_1\le j_0$ be the largest integer such that $\phi$ is not scaling on $D'_{j_1}$. Then  by Corollary \ref{coro:image-ratio},
\begin{equation}\label{equ:ratio-12}
\frac{\mathrm{diam}(\phi^{\ell_{j_1}+1}(D))}{\mathrm{diam}(\phi^{\ell_{j_1}+1}(D'))}\ge\beta|\deg_{c_{j_1}}\phi|_p.
\end{equation}
If $j_1=j_0$, it follows from \eqref{equ:bigger} \eqref{equ:ratio-11}  and \eqref{equ:ratio-12} that
\begin{equation}\label{equ:ratio-13}
\mathrm{diam}(\phi^{\ell'}(D))\ge\beta|\deg_{c_{j_1}}\phi|_p\cdot \mathrm{diam}(\phi^{\ell'}(D'))\ge\beta|\deg_{c_{j_1}}\phi|_p\epsilon_r.
\end{equation}
If $0\le j_1< j_0$, since $\phi$ is not scaling on $\phi^{\ell_{j_1+1}}(D')$, by \eqref{equ:ratio-12} and Corollary \ref{coro:tame-wild} (1),
\begin{align}\label{equ:ratio-14}
\mathrm{diam}(\phi^{\ell_{j_1+1}}(D))&\ge\beta|\deg_{c_{j+1}}\phi|_p\cdot \mathrm{diam}(\phi^{\ell_{j_1+1}}(D'))\nonumber\\
&\ge\beta|\deg_{c_{j_1}}\phi|_p\cdot p^{-\frac{1}{p-1}}\rho(c_{j_1+1},\ \phi^{\ell_{j_1+1}}(D')).
\end{align}
Then the combination of \eqref{equ:ratio-111} and \eqref{equ:ratio-14} implies
$$\frac{\mathrm{diam}(\phi^{\ell_{j_1+1}+1}(D))}{\mathrm{diam}(\phi^{\ell_{j_1+1}+1}(D'))}\ge\beta|\deg_{c_{j_1}}\phi|_p\cdot p^{-\frac{1}{p-1}}|\deg_{c_{j_1+1}}\phi|_p.$$
Inductively, we have
\begin{equation}\label{equ:ratio-15}
\frac{\mathrm{diam}(\phi^{\ell'}(D))}{\mathrm{diam}(\phi^{\ell'}(D'))}\ge\beta\left(p^{-\frac{1}{p-1}}\right)^{j_0-j_1}\prod_{i=j_1}^{j_0}|\deg_{c_i}\phi|_p\ge\beta\left(p^{-\frac{1}{p-1}}\right)^{k_0}\prod_{i=0}^{k_0}|\deg_{c_i}\phi|_p.
\end{equation}
Thus combining \eqref{equ:bigger} and \eqref{equ:ratio-15}, we obtain Claim \ref{c:2}.

\smallskip
Finally, we set
\begin{equation}\label{equ:constant-fix}
\epsilon_\ast:=\beta\left(p^{-\frac{1}{p-1}}\right)^{2\deg\phi-2}\prod_{c\in\mathrm{Crit}_{\mathrm{C}_p\phi)}}|\deg_{c_i}\phi|_p\epsilon_r<\epsilon_r.
\end{equation}
Combining \eqref{equ:same-12}, \eqref{equ:same-12} and Claim \ref{c:2}, we conclude that the statement (2) holds. This completes the proof.
\end{proof}

\begin{remark}\label{rmk:scaling}
Consider the disk $D$ in Lemma \ref{lem:uniform-scaling} and  $\epsilon_\ast>0$ in \eqref{equ:constant-fix}.  Let $D'\subset\mathbb{P}^1_{\mathbb{C}_p}$ be the disk containing $D$ with $\mathrm{diam}(D')=\mathrm{diam}(D)$.  If $D'\cap\mathrm{Crit}_{\mathbb{C}_p}(\phi)=\emptyset$ and $\sup_{n\ge 1}\mathrm{diam}(\phi^n(D))<\epsilon_\ast$, then by  Lemma \ref{lem:uniform-scaling},  there exists a smallest integer $n_0\ge 1$ such that $\phi^{n_0}(D)\cap\mathrm{Crit}_K(\phi)\not=\emptyset$. Pick $c\in \phi^{n_0}(D)\cap\mathrm{Crit}_K(\phi)$. Then by Corollary \ref{coro:tame-wild} (1), there exists a disk $D_1\subset D$ such that $c\in\phi^{n_0}(D_1)$ and $\phi^{n_0}$ is scaling on $D'_1$, where $D'_1\subset D'$ is the disk containing $D_1$ with $\mathrm{diam}(D'_1)=\mathrm{diam}(D_1)$.
\end{remark}

For a point $x\in\mathbb{P}^{1}_{\mathbb{C}_p}$, denote by
$$\mathcal{C}_\phi(x):=\overline{\mathcal{O}_\phi(x)}\cap\mathrm{Crit}_{\mathbb{C}_p}(\phi)$$
the set of critical points in the closure of the forward orbit of $x$.
The following result allows us to only work on the critical points.
\begin{proposition}\label{lem:critical-Julia}
For $x\in J^K_{\C_p}(\phi)$, if $\mathcal{C}_\phi(x)\subset J_K(\phi)$, then $x\in J_K(\phi)$. In particular, if
\begin{equation}\label{equ:inclusion-1}
\mathrm{Crit}^\ast_K(\phi)\subset J_K(\phi),
\end{equation}
then $J^K_{\mathbb{C}_p}(\phi)\subset J_K(\phi)$.
\end{proposition}

\begin{proof}
By Lemma \ref{boundary-julia-lem}, it suffices to consider the case that $x\in\mathrm{Int}(J_{\mathbb{C}_p}^K(\phi))$. Suppose to the contrary that $x\in F_K(\phi)$. Since $F_K(\phi)$ and $\mathrm{Int}(J_{\mathbb{C}_p}^K(\phi))$ are open, there exists $r>0$ such that
$$D_{\mathbb{P}^1_K}(x,r)\subset F_K(\phi)\cap \mathrm{Int}(J_{\mathbb{C}_p}^K(\phi)).$$
Now for any $y\in D_{\mathbb{P}^1_K}(x,r)$ with $y\not=x$, set $s_y=\rho(y,x)$ and consider the disk $D_y:=D_{\mathbb{P}^1_K}(y,s_y)$.
By Lemma \ref{lem:uniform-scaling}, we have two cases:
\begin{enumerate}
\item There exists $n_y\ge 0$ such that
$$\phi^{n_y}(D_y)\cap\mathrm{Crit}_K(\phi)\not=\emptyset.$$
\item  There exists $\epsilon_\ast> 0$, independent of $y$, such that
$$\sup_{n\ge 1}\mathrm{diam}(\phi^n(D_y))\ge\epsilon_\ast.$$
\end{enumerate}

If there exists $y$ arbitrarily close to $x$ such that the case (2) occurs, then $\{\phi^n\}_{n\ge 1}$ is not equicontinuous in any neighborhood in $\mathbb{P}^1_K$ of $x$, which contradicts the assumption  $x\in F_K(\phi)$.

In the remaining case, that is, for any $y$ arbitrarily close to $x$, the case (1) occurs but the case (2) does not occur, we take $n_y\ge 0$ to be the smallest such integer. Note that in this case we have that
$$\emptyset\not=\phi^{n_y}(D_y)\cap\mathrm{Crit}_K(\phi)\subset\mathcal{C}_\phi(x)\subset J_K(\phi).$$
Then by Corollary \ref{coro:scaling-julia} and Remark \ref{rmk:scaling}, the disk $D_y$ contains a point in $J_K(\phi)$, which contradicts the assumption $D_y\subset F_K(\phi)$.

Thus $x\in J_K(\phi)$. Note that for any point $y\in J^K_{\C_p}(\phi)$,
$$\mathcal{C}_\phi(y)\subset \mathrm{Crit}^\ast_K(\phi).$$
The conclusion then follows immediately.
\end{proof}

In the following subsections, we will explore the inclusion \eqref{equ:inclusion-1}. {In particular, we will} show that \eqref{equ:inclusion-1} holds under our assumptions in Theorem \ref{Thm:main}.

\subsection{Equivalence classes for Julia critical points}\label{sec:equivalence}

 In this subsection, we define an equivalence relation on $\mathrm{Crit}^\ast_{\C_p}(\phi)$ and then show under a natural partial order, certain restriction on the minimal elements in the resulting equivalence classes validates the inclusion \eqref{equ:inclusion-1}.

For $c_1,c_2\in \mathrm{Crit}^\ast_{\C_p}(\phi)$, define $c_1\sim c_2$ if $c_1\in\overline{\mathcal{O}_\phi(c_2)}$ and $c_2\in\overline{\mathcal{O}_\phi(c_1)}$.  It follows that $\sim$ is an equivalence relation.
Set
$$\mathcal{M}(\phi):=\mathrm{Crit}^\ast_{\C_p}(\phi)/\sim,$$
and let $\mathcal{Q}:\mathrm{Crit}^\ast_{\C_p}(\phi)\to\mathcal{M}(\phi)$ be the quotient map.
For two classes $[c],[c']\in\mathcal{M}(\phi)$, we define a partial order $\preceq$ as follows: $[c]\preceq[c']$ if $\overline{\mathcal{O}_\phi(c)}\subset\overline{\mathcal{O}_\phi(c')}$, and $[c]\prec[c']$ if $\overline{\mathcal{O}_\phi(c)}\subsetneq\overline{\mathcal{O}_\phi(c')}$. A class $[c]\in\mathcal{M}(\phi)$ is {\it minimal} if $(\mathcal{M}(\phi),\preceq)$ contains no other element less than $[c]$.

\begin{example}
Suppose that $\mathrm{Crit}^\ast_{\C_p}(\phi)=\{c_0,c_1,c_2,c_3\}$ satisfy the following: $\{c_0,c_1,c_2,c_3\}\subset\overline{\mathcal{O}_\phi(c_0)}$, $\{c_1,c_2\}=\overline{\mathcal{O}_\phi(c_1)}\cap\mathrm{Crit}^\ast_{\C_p}(\phi)$, $\{c_1,c_2\}=\overline{\mathcal{O}_\phi(c_2)}\cap\mathrm{Crit}^\ast_{\C_p}(\phi)$, and $\{c_3\}=\overline{\mathcal{O}_\phi(c_3)}\cap\mathrm{Crit}^\ast_{\C_p}(\phi)$. Then $[c_1]=[c_2]\prec [c_0]$ and $[c_3]\prec [c_0]$; and the classes $[c_1]=[c_2]$ and $[c_3]$ are minimal. Moreover, $\mathcal{Q}^{-1}([c_0])=\{c_0\}$.
\end{example}

Note that if $c\in\mathrm{Crit}^\ast_K(\phi)$, then $\overline{\mathcal{O}_\phi(c)}\subset J^K_{\C_p}(\phi)$. It follows that if $c'\in \mathrm{Crit}^\ast_{\C_p}(\phi)$ with $c'\sim c$, then $c'\in \mathrm{Crit}^\ast_K(\phi)$, and hence all the representatives of $[c]$ are in $\mathrm{Crit}^\ast_K(\phi)$. The following result concerning the minimal classes provides a sufficient condition for the inclusion \eqref{equ:inclusion-1}.

\begin{proposition}\label{prop:minimal}
Suppose that for any minimal class $[c]\in\mathcal{M}(\phi)$ with $c\in\mathrm{Crit}^\ast_K(\phi)$, we have $c\in J_K(\phi)$. Then $\mathrm{Crit}^\ast_K(\phi)\subset J_K(\phi)$.
\end{proposition}

\begin{proof}
Pick $c_0\in\mathrm{Crit}^\ast_K(\phi)$. By Lemma \ref{boundary-julia-lem}, it suffices to consider the case that $c_0\in\mathrm{Int}(J^K_{\C_p}(\phi))$. Pick $r>0$ such that $D_{\mathbb{P}^1_K}(c_0,r)\subset \mathrm{Int}(J^K_{\C_p}(\phi))$. Now for any $x\in D_{\mathbb{P}^1_K}(c_0,r)$ with $x\not=c_0$, set $s_x=\rho(x,c_0)$ and consider the disks $D_{\mathbb{P}^1_K}(x,s_x)\subset D_{\mathbb{P}^1_K}(c_0,r)$. By Lemma \ref{lem:uniform-scaling}, we have two cases:
\begin{enumerate}
\item There exists $n_0\ge 0$ such that
$$\phi^{n_0}(D_{\mathbb{P}^1_K}(x,s_x))\cap\mathrm{Crit}_K(\phi)\not=\emptyset.$$
\item  There exists $\epsilon_\ast> 0$, independent of $x$, such that
$$\sup_{n\ge 0}\mathrm{diam}\left(\phi^n(D_{\mathbb{P}^1_K}(x,s_x))\right)\ge\epsilon_\ast.$$
\end{enumerate}

If the case (2) occurs for infinitely many distinct disks $D_{\mathbb{P}^1_K}(x,s_x)$, by the definition of Julia set, we conclude $c_0\in J_K(\phi)$.

Now shrinking $r$ if necessary, we assume that for any disk $D_{\mathbb{P}^1_K}(x,s_x)$, the case (1) occurs but the case (2) does not occur. Pick such a disk $D_{\mathbb{P}^1_K}(x_0,s_{x_0})$ and take $n_0$ to be the smallest integer in the case (1) for $D_{\mathbb{P}^1_K}(x_0,s_{x_0})$. Then there in fact exists $c_1\in\mathcal{C}_\phi(c_0)$ such that
$$\phi^{n_0}\left(D_{\mathbb{P}^1_K}(x_0,s_{x_0})\right)\cap\mathrm{Crit}_K(\phi)=\{c_1\}.$$
Considering the small disks near $c_1$ and applying the above argument repeatedly, we eventually reach a set $\{c_0,c_1,\cdots,c_k\}\subset \mathrm{Crit}^\ast_K(\phi)$ satisfying the following:
\begin{enumerate}
\item[(i)] for $0\le j\le k-1$, the point $c_{j+1}\in\mathcal{C}_\phi(c_j)$ and $[c_k]$ is minimal in $\mathcal{M}(\phi)$; and
\item[(ii)] for $0\le j\le k-1$, there exist integer $n_j\ge 1$ and disk $D_j:=D_{\mathbb{P}^1_K}(x_j, s_{x_j})$ with $x_j\not=c_j$ sufficiently close to $c_j$ and $s_{x_j}=\rho(x_j,c_j)$ such that
\begin{enumerate}
\item $c_{j+1}\in\phi^{n_j}(D_j)$,
\item $\phi^{n_j-1}$ is scaling on $\phi(D_j)$ if $n_j>1$, and
\item $D_{j+1}\subset\phi^{n_j}(D_j)$.
\end{enumerate}
\end{enumerate}
By the assumption in the statement, $c_k\in J_K(\phi)$. Then Corollary \ref{coro:scaling-julia} and Remark \ref{rmk:scaling} imply that $D_{k-1}\cap J_K(\phi)\not=\emptyset$.
Inductively, we conclude
$$D_{\mathbb{P}^1_K}(x_0,  s_{x_0})\cap J_K(\phi)\not=\emptyset.$$
Since we can take $x_0$ arbitrarily close to $c_0$, by the closedness of $J_K(\phi)$, we obtain $c_0\in J_K(\phi)$. Hence the conclusion follows.
\end{proof}

\subsection{Critical points in minimal conjugacy class}\label{sec:minimal}

Now pick $c\in\mathrm{Crit}^\ast_K(\phi)$ such that $[c]\in\mathcal{M}(\phi)$ is a minimal class.  Under the assumptions in Theorem \ref{Thm:main}, this subsection is devoted to proving that $c\in J_K(\phi)$, which verifies the assumption in Proposition \ref{prop:minimal}. If $c\in\partial J^K_{\C_p}(\phi)$, then by Lemma \ref{boundary-julia-lem}, we have $c\in J_K(\phi)$. In what follows, we assume that $c\in\mathrm{Int}(J^K_{\C_p}(\phi))$. Recall the quantity $r_\ast$ from Section \ref{sec:distance}.
Fix a small $0<r<r_\ast$ and pick an $\epsilon_r>0$ satisfying Corollary \ref{coro:tame-wild} (2). Set $\epsilon'_r:=\min\{r,\epsilon_r\}$ and define $\epsilon_0:=\min\{\epsilon'_r/p,\epsilon_\ast\}$, where $\epsilon_\ast$ is a fixed constant satisfying Lemma \ref{lem:uniform-scaling}. Shrinking  $\epsilon_0$ if necessary, we can assume
\begin{equation}\label{equ:empty}
\phi\left(D_{\mathbb{P}^1_{\mathbb{C}_p}}(c,\epsilon_0)\right)\not=\mathbb{P}^1_{\C_p}.
\end{equation}

\subsubsection{Non-recurrent case}
Fix the notations as above. We consider the case that the critical point $c$ is non-recurrent.
\begin{proposition}\label{prop:non-recurrent}
If the critical point $c$ is non-recurrent, then $c\in J_K(\phi)$.
\end{proposition}
\begin{proof}
Suppose to the contrary that $c\in F_K(\phi)$. Since $c\in\mathrm{Int}(J^K_{\C_p}(\phi))$, consider a disk $D_{\mathbb{P}^1_K}(c,\delta)\subset\mathrm{Int}(J^K_{\C_p}(\phi))\cap F_K(\phi)$.
By the equicontiunity, shrinking $\delta>0$ if necessary,  we can assume that for all $n\ge 0$,
\begin{equation}\label{equ:small}
\mathrm{diam}\left(\phi^n(D_{\mathbb{P}^1_K}(c,\delta))\right)<\epsilon_0.
\end{equation}
Pick an arbitrary $x\in D_{\mathbb{P}^1_K}(c,\delta)$ with $x\not=c$, and set $s_x=\rho(x,c)$. For the disk $D(x):=D_{\mathbb{P}^1_K}(x,s_x)\subset D_{\mathbb{P}^1_K}(c,\delta)$, by the choice of $\epsilon_0$, the orbit of $D(x)$ disjoints the critical points. Then it follows from Lemma \ref{lem:uniform-scaling} that
$$\sup_{n\ge 1}\mathrm{diam}(\phi^n(D(x)))\ge\epsilon_\ast,$$
which contradicts \eqref{equ:small} since $\epsilon_0\le\epsilon_\ast$.
Hence the conclusion follows.
\end{proof}

\subsubsection{Recurrent case}\label{sec:recurrent}
Now we work on the case that $c$ is recurrent. By the no wild recurrent critical points assumption, the critical point $c$ is tame. Let $c'\in\mathrm{Crit}^\ast_K(\phi)$ be any critical point such that $[c']=[c]$. If $c'\in\partial J^K_{\C_p}(\phi)$, by Lemma \ref{boundary-julia-lem} and the forward invariance and closedness of $J_K(\phi)$, we have $c\in J_K(\phi)$.
Thus for the remainder of this section, we assume that $c'\in\mathrm{Int}(J^K_{\C_p}(\phi))$ for all such $c'$. Consider $0<\delta<\epsilon_0$ such that for all aforementioned $c'$, the disk $D_{\mathbb{P}^1_K}(c',\delta)\subset\mathrm{Int}(J^K_{\C_p}(\phi))$.  Moreover, for $x\in D_{\mathbb{P}^1_K}(c',\delta)$ with $x\not=c'$, set $s_x=\rho(x,c')$, let $D(x):=D_{\mathbb{P}^1_K}(x,s_x)\subset D_{\mathbb{P}^1_K}(c',\delta)$ and $\overline{D}(x):=\overline{D}_{\mathbb{P}^1_K}(x,s_x)$, denote $D_{\C_p}(x):=D_{\mathbb{P}^1_{\C_p}}(x,s_x)$ and $\overline{D}_{\C_p}(x):=\overline{D}_{\mathbb{P}^1_{\C_p}}(x,s_x)$ the corresponding disks in $\mathbb{P}^1_{\C_p}$, and write
$$S_K(x):=\{y\in D_{\mathbb{P}^1_K}(c',\delta): s_y=s_x\}$$
the sphere near $c'$ containing $x$.

We first show a result on the existence of the repelling periodic points, which will imply that $c\in J_K(\phi)$ in the case that $\mathcal{Q}^{-1}([c])=\{c\}$.

\begin{proposition}\label{prop:repelling}
Pick $x\in D_{\mathbb{P}^1_K}(c,\delta)$ with $x\not=c$. Assume that there exists a smallest $n:=n_x\ge 0$ such that $\phi^n(D(x))\cap\mathrm{Crit}_K(\phi)\not=\emptyset$. If
\begin{equation}\label{equ:c}
\phi^n(D(x))\cap\mathrm{Crit}_K(\phi)=\{c\}
\end{equation}
and for all $1\le m\le n$,
\begin{equation}\label{equ:less}
\mathrm{diam}\left(\phi^m(\overline{D}(x))\right)<\epsilon_0,
\end{equation}
then $D(x)$ contains a repelling periodic point of $\phi$ with period $n$.
\end{proposition}

\begin{proof}
To ease notations, write $D:=D(x)$ and $\overline{D}:=\overline{D}(x)$. Consider the disks $D_{\C_p}:=D_{\C_p}(x)$ and $\overline{D}_{\C_p}:=\overline{D}_{\C_p}(x)$ in $\mathbb{P}^1_{\C_p}$. Let $n_0=0$ and inductively define $n_{j-1}<n_j\le n$ to be the smallest integer such that $c\in\phi^{n_j}(\overline{D}_{\C_p})$. By \eqref{equ:c}, there exists $j_0\ge 1$ such that $n_{j_0}=n$. The existence of repelling periodic points in $D$ is a consequence of the following claim.
\begin{claim}\label{claim:repelling}
For $1\le j\le j_0$, the following hold:
 \begin{enumerate}
 \item[(i)] $c\not\in\phi^{n_{j-1}}(D_{\C_p})$.
 \item[(ii)] $\phi^{n_{j-1}+1}(\overline{D}_{\C_p})\not=\mathbb{P}^1_{\C_p}$.
  \item[(iii)] $\phi^{n_j-n_{j-1}-1}$ is scaling on $\phi^{n_{j-1}+1}(\overline{D}_{\C_p})$.
 \item[(iv)] $\phi^{n_j}$ is scaling on $D_{\C_p}$.
 \item[(v)] $\mathrm{diam}(\phi^{n_j}(\overline{D}_{\C_p}))=\mathrm{diam}(\phi^{n_j}(\overline{D}))<\epsilon_0$ if $j_0>1$.
 \end{enumerate}
\end{claim}

We proceed the proof of Claim \ref{claim:repelling} by induction.
For $j=1$, the statement (i) follows immediately and the statement (ii) follows from \eqref{equ:empty}. We begin to show the statements (iii)--(v) for $j=1$. Noting that there is no open disk in $\phi(\overline{D}_{\C_p})$ containing $\phi(\overline{D})$, we obtain
\begin{equation}\label{equ:diam-same-1}
\mathrm{diam}(\phi(\overline{D}_{\C_p}))=\mathrm{diam}(\phi(\overline{D})).
\end{equation}
Then by \eqref{equ:c} and \eqref{equ:less}, the map $\phi$ is scaling on $\phi(\overline{D}_{\C_p})$. Let $1\le i_1\le n_1-1$ be the largest integer such that $\phi^{i_1}$ is scaling on $\phi(\overline{D}_{\C_p})$. If $i_1\not=n_1-1$, then by \eqref{equ:diam-same-1}, we have
\begin{equation}\label{equ:diam-same-i}
\mathrm{diam}(\phi^{i_1+1}(\overline{D}_{\C_p}))=\mathrm{diam}(\phi^{i_1+1}(\overline{D})),
\end{equation}
 but
 \begin{equation}\label{equ:diam-same-i+1}
 \mathrm{diam}(\phi^{i_1+2}(\overline{D}_{\C_p}))>\mathrm{diam}(\phi^{i_1+2}(\overline{D})).
 \end{equation}
 However, by \eqref{equ:c}, \eqref{equ:less} and \eqref{equ:diam-same-i}, the map $\phi$ is scaling on $\phi^{i_1+1}(\overline{D}_{\C_p})$, which implies that
 $$\mathrm{diam}(\phi^{i_1+2}(\overline{D}_{\C_p}))=\mathrm{diam}(\phi^{i_1+2}(\overline{D})).$$
 This contradicts \eqref{equ:diam-same-i+1}. Hence $i_1=n_1-1$ and $\phi^{n_1-1}$ is scaling on $\phi(\overline{D}_{\C_p})$. Thus the statement (iii) holds for $j=1$.
 Since by Corollary \ref{coro:tame-wild} (1), the map $\phi$ is scaling on $D_{\C_p}$, we in fact have that $\phi^{n_1}$ is scaling on $D_{\C_p}$, and hence the statement (iv) holds for $j=1$. Applying \eqref{equ:c}, \eqref{equ:less}, \eqref{equ:diam-same-1} and the statement (iii) for $j=1$, we obtain the statement (v) for $j=1$.

If $j_0=1$, then we are done. If $j_0>1$, we do the induction step.
We assume that the statements (i)--(v) hold for $j=k<j_0$ and we show that the statements (i)--(v) hold for $j=k+1$.
For the statement (i), if $c\in\phi^{n_k}(D_{\C_p})$, then $c\in\phi^{n_k}(D)$ since $\phi^{n_k}$ is scaling on $D_{\C_p}$, which contradicts the choice of $n$. Thus the statement (i) holds for $j=k+1$. For the statement (ii), since $\phi^{n_k-n_{k-1}-1}$ is scaling on $\phi^{n_{k-1}+1}(\overline{D}_{\C_p})\not=\mathbb{P}^1_{\C_p}$, it follows that $\phi^{n_k}(\overline{D}_{\C_p})\not=\mathbb{P}^1_{\C_p}$. Moreover, noting that $\mathrm{diam}(\phi^{n_k}(\overline{D}_{\C_p}))<\epsilon_0$ and $c\in\phi^{n_k}(\overline{D}_{\C_p})$, by \eqref{equ:empty}, we have $\phi^{n_k+1}(\overline{D}_{\C_p})\not=\mathbb{P}^1_{\C_p}$. Hence the statement (ii) holds for $j=k+1$.

Since $\mathrm{diam}(\phi^{n_k}(\overline{D}_{\C_p}))=\mathrm{diam}(\phi^{n_k}(\overline{D}))<\epsilon_0$ and $c\in\phi^{n_k}(\overline{D})$, as in \eqref{equ:diam-same-1}, we have
 \begin{equation}\label{equ:diam-same-1-k}
\mathrm{diam}(\phi^{n_k+1}(\overline{D}_{\C_p}))=\mathrm{diam}(\phi^{n_k+1}(\overline{D})).
\end{equation}
Moreover, by \eqref{equ:c} and \eqref{equ:less}, the map $\phi$ is scaling on $\phi^{n_k+1}(\overline{D}_{\C_p})$. Let $1\le i_{k+1}\le n_{k+1}-n_k-1$ be the largest integer such that $\phi^{i_{k+1}}$ is scaling on $\phi^{n_k+1}(\overline{D}_{\C_p})$. If $i_1\not=n_{k+1}-n_k-1$, then by \eqref{equ:diam-same-1-k}, it follows that
\begin{equation}\label{equ:diam-same-i-k}
\mathrm{diam}(\phi^{n_k+i_{k+1}+1}(\overline{D}_{\C_p}))=\mathrm{diam}(\phi^{n_k+i_{k+1}+1}(\overline{D})),
\end{equation}
 but
 \begin{equation}\label{equ:diam-same-k+1}
 \mathrm{diam}(\phi^{n_k+i_{k+1}+2}(\overline{D}_{\C_p}))>\mathrm{diam}(\phi^{n_k+i_{k+1}+2}(\overline{D})).
 \end{equation}
 However, by \eqref{equ:c}, \eqref{equ:less} and \eqref{equ:diam-same-i-k}, the map $\phi$ is scaling on $\phi^{n_k+i_{k+1}+1}(\overline{D}_{\C_p})$, which implies
 $$\mathrm{diam}(\phi^{n_k+i_{k+1}+2}(\overline{D}_{\C_p}))=\mathrm{diam}(\phi^{n_k+i_{k+1}+2}(\overline{D})).$$
 This contradicts \eqref{equ:diam-same-k+1}. Hence $i_{k+1}=n_{k+1}-n_k-1$ and $\phi^{n_{k+1}-n_k-1}$ is scaling on $\phi^{n_k+1}(\overline{D}_{\C_p})$. Thus the statement (iii) holds for $j=k+1$.

By Corollary \ref{coro:tame-wild} (1),  since $\mathrm{diam}(\phi^{n_k}(\overline{D}_{\C_p}))<\epsilon_0$, the statement (i) for $j=k+1$ implies that $\phi$ is scaling on $\phi^{n_k}(D_{\C_p})$. Moreover, the statement (iii) for $j=k+1$ implies that $\phi^{n_{k+1}-n_k-1}$ is scaling on $\phi^{n_k+1}(D_{\C_p})$. It follows that $\phi^{n_{k+1}-n_k}$ is scaling on $\phi^{n_k}(D_{\C_p})$. Since  $\phi^{n_k}$ is scaling on $D_{\C_p}$, the statement (iv) holds for $j=k+1$.
Combining the statement (iii) for $j=k+1$ and \eqref{equ:diam-same-1-k}, we obtain
\begin{equation}\label{equ:diam-same-n-k+1}
\mathrm{diam}(\phi^{n_{k+1}}(\overline{D}_{\C_p}))=\mathrm{diam}(\phi^{n_{k+1}}(\overline{D})).
\end{equation}
Then the statement (v) for $j=k+1$ follows immediately from \eqref{equ:less} and \eqref{equ:diam-same-n-k+1}.
Therefore, Claim \ref{claim:repelling} holds.

\smallskip
We continue to show that $D$ contains a repelling periodic point of $\phi$.
From the statements (ii) and (iii) for $j=j_0$ in Claim \ref{claim:repelling}, we conclude that $\phi^{n}(\overline{D}_{\C_p})\not=\mathbb{P}^1_{\C_p}$.
 It follows that
\begin{equation}\label{equ:diam-same}
\mathrm{diam}(\phi^{n}(\overline{D}_{\C_p}))=\mathrm{diam}(\phi^{n}(D_{\C_p})).
\end{equation}
Note that
\begin{equation}\label{equ:ratio}
\frac{\mathrm{diam}(D)}{\mathrm{diam}(D_{\C_p})}=|\pi|.
\end{equation}
Then the statement (iv) for $j=j_0$ in Claim \ref{claim:repelling} implies that
\begin{equation}\label{equ:ratio-n}
\frac{\mathrm{diam}(\phi^n(D))}{\mathrm{diam}(\phi^n(D_{\C_p}))}=|\pi|.
\end{equation}
Since $c\in J_{\C_p}(\phi)$ and $c\in\overline{D}\cap\phi^n(D)\subset \overline{D}_{\C_p}\cap\phi^n(\overline{D}_{\C_p})$, it follows that $\overline{D}_{\C_p}\subsetneq\phi^n(\overline{D}_{\C_p})$. Hence by \eqref{equ:diam-same},
\begin{equation}\label{equ:ratio-less}
\mathrm{diam}(D_{\C_p})<\mathrm{diam}(\phi^n(D_{\C_p})).
\end{equation}
Combining \eqref{equ:ratio}, \eqref{equ:ratio-n} and \eqref{equ:ratio-less}, we obtain
\begin{equation*}
\mathrm{diam}(D)=|\pi|\mathrm{diam}(D_{\C_p})<|\pi|\mathrm{diam}(\phi^n(D_{\C_p}))=\mathrm{diam}(\phi^n(D)).
\end{equation*}
Hence
\begin{equation}\label{equ:ratio-less-K}
\mathrm{diam}(\overline{D})\le\mathrm{diam}(\phi^n(D)).
\end{equation}
Again since $c\in\overline{D}\cap\phi^n(D)$, it follows from \eqref{equ:ratio-less-K} that $\overline{D}\subset\phi^n(D)$, and hence $D\subsetneq\phi^n(D)$. Then by the statement (iv) for $j=j_0$ in Claim \ref{claim:repelling}, the map $\phi$ has a repelling periodic point $x\in D$. By the choice of $n$, the point $x$ is of period $n$.
\end{proof}

\begin{corollary}\label{coro:1}
If $\mathcal{Q}^{-1}([c])=\{c\}$ and $c$ is recurrent, then $c\in J_K(\phi)$.
\end{corollary}
\begin{proof}
Suppose to the contrary that $c\in F_K(\phi)$. Shrinking $\delta$ if necessary, we assume $D_{\mathbb{P}^1_K}(c,\delta)\subset F_K(\phi)$. Moreover, by the equicontiunity, for all $m\ge 0$, we can assume
$$\mathrm{diam}\left(\phi^m(D_{\mathbb{P}^1_K}(c,\delta))\right)<\epsilon_0.$$
Pick an arbitrary $x\in D_{\mathbb{P}^1_K}(c,\delta)$ with $x\not=c$ and consider $D(x)$. Then $D(x)\subset\mathrm{Int}(J_K(\phi))$. If the orbit of $D(x)$ intersects $\mathrm{Crit}_K(\phi)$, by the choice of $\epsilon_0$, we have that $c$ is the unique point in the intersection. Proposition \ref{prop:repelling} implies that $D(x)$ contains a repelling periodic point of $\phi$. Since $D(x)\subset F_K(\phi)$, the orbit of $D(x)$ must disjoint $\mathrm{Crit}_K(\phi)$. However, noting that $D(x)\subset\mathrm{Int}(J_K(\phi))$, by Lemma \ref{lem:uniform-scaling}, we deduce that there exists $\epsilon_\ast>0$, independent of $x$ such that
$$\sup_{n\ge 1}\mathrm{diam}\phi^n(D(x))\ge\epsilon_\ast.$$
This implies $c\in J_K(\phi)$ which hence contradicts the assumption $c\in F_K(\phi)$.
\end{proof}

Now we consider the case that $\mathcal{Q}^{-1}([c])=\{c,c_1\}$ with $c\not=c_1$. Again, we first show a result on the existence of repelling periodic points.

\begin{proposition}\label{prop:repelling-2}
 Pick $x\in D_{\mathbb{P}^1_K}(c,\delta)$ with $x\not=c$. Suppose that for any $y\in S_K(x)$, there exists a smallest $n_y\ge 1$ such that  $\phi^{n_y}(D(y))\cap\mathrm{Crit}_K(\phi)\not=\emptyset$, and further assume that 
 \begin{enumerate}
\item $\phi^{n_y}(D(y))\cap\mathrm{Crit}_K(\phi)=\{c_1\},$
and
\item for all $1\le m\le n_y$, $\mathrm{diam}\left(\phi^m(\overline{D}(y))\right)<\epsilon_0.$
\end{enumerate}
Then the following hold.
\begin{enumerate}
\item[(i)] If there exists $y_0\in S_K(x)$ such that $n_{y_0}\not=n_{x}$, then $D(x)\cup D(y_0)$ contains a repelling preperiodic point of $\phi$.
\item[(ii)] If $n_{y}=n_{x}$ for all $y\in S_K(x)$, picking $z\in D_{\mathbb{P}^1_K}(c_1,\epsilon_0)$ with $s_z=|\pi|\mathrm{diam}(\phi^{n_x}(D_{\C_p}(x)))$, we suppose that for any $w\in S_K(z)$, there exists a smallest $n_w\ge 1$ such that $\phi^{n_w}(D(w))\cap\mathrm{Crit}_K(\phi)\not=\emptyset$, and  further assume
\begin{enumerate}
\item  $\phi^{n_w}(D(w))\cap\mathrm{Crit}_K(\phi)=\{c\},$
and
\item for all $1\le m\le n_w$, $\mathrm{diam}\left(\phi^m(\overline{D}(w))\right)<\epsilon_0.$
\end{enumerate}
Then
 if $\phi^{n_x}(D_{\C_p}(x))=\phi^{n_x}(D_{\mathbb{P}^1_{\C_p}}(c,s_x))$ and  $n_w=n_z$ for all $w\in S_K(z)$, we have that for any $y\in S_K(x)$, the disk $D(y)$ contains a repelling periodic point.
\end{enumerate}
\end{proposition}
\begin{proof}
If there exists $y_0\in S_K(x)$ such that $n_{y_0}\not=n_{x}$, without loss of generality, we may assume that $n_{y_0}>n_x$. Observe $c_1\in\phi^{n_x}(D(x))$. Then, by  the assumption (2), $\phi^{n_x}(D(x))\subsetneq D_{\mathbb{P}^1_K}(c_1,\epsilon_0)$, and we have that $c_1\not\in\phi^{n_x}(D(y_0))\subset D_{\mathbb{P}^1_K}(c_1,\epsilon_0)$ but $c_1\in\phi^{n_{y_0}-n_x}(\phi^{n_x}(D(x))$. Note that by the choice of $\epsilon_0$ and the assumption (2), the map $\phi^{n_{y_0}-n_x}$ is scaling on the smallest disk in $\mathbb{P}^1_{\C_p}$ containing $\phi^{n_x}(D(y_0))$. Then $\phi^{n_x}(D(y_0))$ contains a repelling periodic point of $\phi$ with period $n_{y_0}-n_x$. Since $\phi^{n_x}$ is scaling on $D_{\C_p}(y_0)$, Corollary \ref{coro:scaling-julia} implies that $D(y_0)$ contains a repelling preperiodic point of $\phi$, and hence the statement (i) holds.

Now we show the statement (ii). Note that by Corollary \ref{coro:tame-wild} (1) and  \eqref{equ:empty},
\begin{equation}\label{equ:equ}
\mathrm{diam}(\phi(D(x)))=\mathrm{diam}(\phi(D_{\C_p}(x)))=\mathrm{diam}(\phi(\overline{D}_{\C_p}(x))).
\end{equation}
By  \eqref{equ:equ} and the assumption (2), we have $\phi(\overline{D}_{\C_p}(x))\not=\mathbb{P}^1_{\C_p}$. The choice of $\epsilon_0$ and the assumption (2) imply that $\phi^{n_x-1}$ is scaling on $\phi(\overline{D}_{\C_p}(x))$ and hence $\phi^{n_x}(\overline{D}_{\C_p}(x))\not=\mathbb{P}^1_{\C_p}$.

For $y\in S_K(x)$, consider the disks $\overline{D}^{(1)}(y):=\overline{D}_{\mathbb{P}^1_K}(y,|\pi|s_y)$ and $\overline{D}^{(1)}_{\C_p}(y):=\overline{D}_{\mathbb{P}^1_{\C_p}}(y,|\pi|s_y)$, and let $\overline{D}^{(1)}_{\C_p}(c):=\overline{D}_{\mathbb{P}^1_{\C_p}}(c,|\pi|s_x)$. Then we conclude that
\begin{equation}\label{equ:ration-y}
\frac{\mathrm{diam}\left(\phi^{n_x}(\overline{D}^{(1)}_{\C_p}(y))\right)}{\mathrm{diam}\left(\phi^{n_x}(\overline{D}_{\C_p}(y))\right)}=|\pi|,
\end{equation}
and
\begin{equation}\label{equ:ration-c}
\frac{\mathrm{diam}\left(\phi^{n_x}(\overline{D}^{(1)}_{\C_p}(c))\right)}{\mathrm{diam}\left(\phi^{n_x}(\overline{D}_{\C_p}(y))\right)}=|\pi|^{\deg_c\phi}.
\end{equation}

 If $\phi^{n_x}(D_{\C_p}(x))=\phi^{n_x}(D_{\mathbb{P}^1_{\C_p}}(c,s_x))$,  then $c_1\in\phi^{n_x}(D_{\mathbb{P}^1_{\C_p}}(c,s_x))$.
 From the assumption (2) and the assumption $n_y=n_x$, we deduce that $c_1\in\phi^{n_x}(\overline{D}^{(1)}_{\C_p}(y))$. By \eqref{equ:ration-y} and \eqref{equ:ration-c},
\begin{equation}\label{equ:order}
\phi^{n_x}(\overline{D}^{(1)}_{\C_p}(c))\subsetneq \phi^{n_x}(\overline{D}^{(1)}_{\C_p}(y)).
\end{equation}
Furthermore, by the choice of $z$ and by the combination of \eqref{equ:ration-y} and \eqref{equ:order},  we have
\begin{equation*}
\mathrm{diam}\left(\phi^{n_x}(\overline{D}^{(1)}_{\C_p}(y))\right)=s_z.
\end{equation*}

If further $n_w=n_z$ for all $w\in S_K(z)$, then $c\in\phi^{n_x+n_z}(\overline{D}^{(1)}_{\C_p}(y))$. Moreover,  \begin{equation}\label{equ:above}
\overline{D}_{\C_p}(y)\subset\phi^{n_x+n_z}\left(\overline{D}^{(1)}_{\C_p}(y)\right).
\end{equation}
Indeed, for otherwise, if $\phi^{n_x+n_z}(\overline{D}^{(1)}_{\C_p}(y))\subsetneq \overline{D}_{\C_p}(y)$, then $\phi^{n_x+n_z}(\overline{D}^{(1)}_{\C_p}(y))\subset \overline{D}^{(1)}_{\C_p}(c)$, and hence by \eqref{equ:order}, we obtain $\phi^{n_x}(\overline{D}^{(1)}_{\C_p}(c))\subsetneq\overline{D}^{(1)}_{\C_p}(c)$, which implies that $\overline{D}^{(1)}_{\C_p}(c)\subset F_{\C_p}(\phi)$ contradicting the assumption $c\in J_{\C_p}(\phi)$.

By \eqref{equ:above}, we first consider the case that $\overline{D}_{\C_p}(y)=\phi^{n_x+n_z}(\overline{D}^{(1)}_{\C_p}(y))$. If $c\in\phi^{n_x+n_z}(\overline{D}^{(1)}_{\C_p}(c))$, then by \eqref{equ:order},  we have $\phi^{n_x+n_z}(\overline{D}^{(1)}_{\C_p}(c))\subset\overline{D}_{\C_p}(y)$, and hence $\phi^{n_x+n_z}(\overline{D}^{(1)}_{\C_p}(c))\subset\overline{D}^{(1)}_{\C_p}(c)$, which also contradicts the assumption $c\in J_{\C_p}(\phi)$. If $c\not\in\phi^{n_x+n_z}(\overline{D}^{(1)}_{\C_p}(c))$, then by the assumptions (a) and (b), we have
\begin{equation}\label{equ:ration-c-1}
\frac{\mathrm{diam}\left(\phi^{n_x+n_z}(\overline{D}^{(1)}_{\C_p}(c))\right)}{\mathrm{diam}\left(\phi^{n_x+n_z}(\overline{D}_{\C_p}(y))\right)}=|\pi|^{\deg_{c_1}\phi}.
\end{equation}
 It follows that there exists $y_0\in S_K(x)$ such that
 \begin{equation}\label{equ:bound}
 \phi^{n_x+n_z}\left(\overline{D}^{(1)}_{\C_p}(c)\right)\subsetneq\overline{D}^{(1)}_{\C_p}(y_0)\subset\overline{D}_{\C_p}(x).
 \end{equation}
 Inductively, by \eqref{equ:bound}, for any $k\ge 1$, we conclude that $$\phi^{k(n_x+n_z)}\left(\overline{D}^{(1)}_{\C_p}(c)\right)\subsetneq\overline{D}_{\C_p}(x),$$
 which implies that
  \begin{equation}\label{equ:F}
  \phi^{n_x+n_z}(\overline{D}^{(1)}_{\C_p}(c))\subset F_{\C_p}(\phi).
  \end{equation}
  However, noting $\overline{D}(x)\subset D_{\mathbb{P}^1_K}(c,\delta)\subset J_{\C_p}^K(\phi)$, observing $ \phi^{n_x+n_z}(\overline{D}^{(1)}_{\C_p}(c))\cap\mathbb{P}^1_K\not=\emptyset$ since $c\in\mathbb{P}^1_K$, and applying  \eqref{equ:bound}, we obtain
 $$\phi^{n_x+n_z}\left(\overline{D}^{(1)}_{\C_p}(c)\right)\cap J_{\C_p}(\phi)\not=\emptyset,$$
 which contradicts \eqref{equ:F}.

 We now consider the case that $\overline{D}_{\C_p}(y)\subsetneq\phi^{n_x+n_z}(\overline{D}^{(1)}_{\C_p}(y))$. Then by the assumptions (a), (b) and Corollary \ref{coro:tame-wild} (1), for any $w\in S_K(z)$,
  \begin{equation}\label{equ:contain}
  \overline{D}_{\C_p}(y)\subsetneq\phi^{n_z}(D_{\C_p}(w)).
  \end{equation}
 Since by the assumptions (1), (2) and Corollary \ref{coro:tame-wild} (1), the map $\phi^{n_x}$ is scaling on $\overline{D}^{(1)}_{\C_p}(y)$, considering the disks $D^{(1)}_{\C_p}(y'):=D_{\mathbb{P}^1_{\C_p}}(y',|\pi|s_y)$ for $y'\in\overline{D}^{(1)}_{\C_p}(y)$ with $\rho(y,y')=|\pi|s_y$, we have that $\phi^{n_x}$ maps all except one such disks to maximal disks in $S_K(z)$. Pick such a disk $D^{(1)}_{\C_p}(y')$ that is mapped to $S_K(z)$ by $\phi^{n_x}$. It follows that $\phi^{n_x+n_z}$ is scaling on $D^{(1)}_{\C_p}(y')$. By \eqref{equ:contain},
$$D^{(1)}_{\C_p}(y')\subsetneq \overline{D}_{\C_p}(y)\subsetneq\phi^{n_x+n_z}(D^{(1)}_{\C_p}(y')).$$
and hence setting $D^{(1)}(y'):=D^{(1)}_{\C_p}(y')\cap\mathbb{P}^1_K$, we conclude
$$D^{(1)}(y')\subsetneq \phi^{n_x+n_z}(D^{(1)}(y')).$$
It follows that $D^{(1)}(y')$ contains a periodic point of $\phi$.

Thus the statement (ii) holds and this completes the proof.
\end{proof}

\begin{corollary}\label{coro:2}
If $\mathcal{Q}^{-1}([c])=\{c, c_1\}$ and $c\not=c_1$, then $c\in J_K(\phi)$.
\end{corollary}

\begin{proof}
Suppose to the contrary that $c\in F_K(\phi)$. It follows that $c_1\in F_K(\phi)$; indeed, for otherwise, $\overline{\mathcal{O}_\phi(c_1)}\subset J_K(\phi)$ and hence $c\in J_K(\phi)$, which is a contradiction. Shrinking $\delta$ if necessary, we can assume $D_{\mathbb{P}^1_K}(c,\delta)\cup D_{\mathbb{P}^1_K}(c_1,\delta)\subset F_K(\phi)$. Moreover, by the equicontiunity, for all $m\ge 0$, we can also assume that
$$\mathrm{diam}\left(\phi^m(D_{\mathbb{P}^1_K}(c,\delta))\right)+\mathrm{diam}\left(\phi^m(D_{\mathbb{P}^1_K}(c_1,\delta))\right)<\epsilon_0.$$
Applying Lemma \ref{lem:uniform-scaling}, Propositions \ref{prop:repelling}, and \ref{prop:repelling-2}, and switching $c$ and $c_1$ if necessary, we only need to dispose of the case that for all $x\in D_{\mathbb{P}^1_K}(c,\delta)$ with $x\not=c$, one has
\begin{enumerate}
\item[(i)]  the assumptions (1)--(3) in Proposition \ref{prop:repelling-2} hold,
\item[(ii)]  $n_{y}=n_{x}$ for all $y\in S_K(x)$,
\item[(iii)]  $\phi^{n_x}(D_{\C_p}(x))\not=\phi^{n_x}(D_{\mathbb{P}^1_{\C_p}}(c,s_x))$,
\item [(iv)] $\phi^{n_x}(D_{\C_p}(x))\subset D_{\mathbb{P}^1_{\C_p}}(c_1,\delta)$, and
\item[(v)]  there exists a smallest integer $\ell\ge 1$ such that $\phi^{n_x+\ell}(D_{\mathbb{P}^1_{\C_p}}(c,s_x))\cap\mathrm{Crit}_{\C_p}(\phi)\not=\emptyset$, and furthermore,
\begin{enumerate}
\item $\phi^{n_x+\ell}(D_{\mathbb{P}^1_{\C_p}}(c,s_x))\cap\mathrm{Crit}_{\C_p}(\phi)=\{c\}$, and
\item $c\not\in\phi^{n_x+\ell}(D_{\C_p}(x))$.
 \end{enumerate}
\end{enumerate}
Indeed, in all other cases, we have that either $D_{\mathbb{P}^1_K}(c,\delta)$ contains a repelling (pre)periodic point or $\sup_{n\ge 0}\mathrm\{\phi^n(D(x))\}\ge\min\{\epsilon_\ast, \delta\}$, which contradicts the assumption $D_{\mathbb{P}^1_K}(c,\delta)\subset F_K(\phi)$.

It follows from (i)-(v) that
$$\rho\left(\phi^{n_x+\ell}(c),\ \phi^{n_x+\ell}(x)\right)>\rho(c,x).$$
Considering the forward orbit of $D(x)$ in  $D(c,\delta)$ and repeatedly applying the above argument, we eventually obtain
$$\sup_{n\ge 1}\rho\left(\phi^{n}(c),\phi^{n}(x)\right)>\delta.$$
Hence $\{\phi^n\}_{n\ge 1}$ is not equicontinuous at $c$ in $\P^1_K$, which contradicts the assumption $c\in F_K(\phi)$.
\end{proof}

\begin{proof}[Proof of Theorem \ref{Thm:main}]
Theorem \ref{Thm:main} immediately follows from Lemma \ref{boundary-julia-lem}, Propositions \ref{lem:critical-Julia}, \ref{prop:minimal}, \ref{prop:non-recurrent} and Corollaries \ref{coro:1}, \ref{coro:2}. \end{proof}

\section{A first example} \label{sec:example}
In this section, consider $f:\mathbb{P}^1_{\mathbb{Q}_2}\to\mathbb{P}^1_{\mathbb{Q}_2}$ given by
$$f(x)=\frac{9}{4}x(x-1)^2.$$
As we will see in Section \ref{sec:dynamics} that the polynomial $f$ is postcritically finite, that is, each critical point of $f$ is eventually periodic. We prove that Theorem \ref{Thm:GF} holds for $f$.
The argument sheds lights on the proof of Theorem \ref{Thm:GF} in the next section.

\begin{remark}
In \cite[Section 4]{Ingram12}, Ingram classified the cubic postcritically finite complex polynomials with rational coefficients into finitely many conjugacy classes. Since these polynomials are of rational coefficients, we can regard them as polynomials from $\mathbb{P}^1_{\mathbb{Q}_p}$ to itself. It turns out that for $p\geq 3$,  these polynomials possess no Julia critical points, and up to conjugacy, the above $f$ is the unique such a polynomial whose corresponding $2$-adic Julia set contains a critical point.
\end{remark}

Our main goal in this section is to prove the following result. Recall the notations from Section \ref{sec:intro} and write $I(f):=I_{\Q_2}(f)=J_{\mathbb{Q}_2}(f) \setminus \mathrm{GO}_{\Q_2}(1)$ for simplicity.
\begin{theorem}\label{Thm:example}
Let $f$ be as above. Then there exist a countable state Markov shift $(\Sigma_A,\sigma)$ and a bijection $h: J_{\mathbb{Q}_2}(f)\to\Sigma_A$ such that $(I(f),f)$ is topologically conjugate to $(h(I(f)), \sigma)$ via $h$.
\end{theorem}

In Section \ref{sec:dynamics},
we describe the orbit of disks in $\mathbb{Q}_2$ under $f$ and depict the Julia and Fatou sets of $f$. In Section \ref{sec:tree}, we illustrate the orbit of disks on $\mathbb{Q}_2$ under $f$ in terms of tree of closed disks. Then we prove Theorem \ref{Thm:example} in Section \ref{sec:example-proof}.  The proof contains exact scaling ratios of the iterations of $f$. {However, we remark that} when we deal with the general case in Section \ref{sec:GF-maps}, we do not require those ratios precisely. To help readers catch the idea of the proof of Theorem \ref{Thm:GF}, all details for $f$ will be given. Finally, we compute the Gurevich entropy of $f$ in Section \ref{sec:Gurevich}. It turns out that this entropy is the logarithm of an algebraic number.

To ease notations, in this section, we write $|\cdot|$ for the absolute value on $\mathbb{Q}_2$ and write $F(f)$, $J(f)$ and $\mathrm{GO}(1)$ for the Fatou set $F_{\mathbb{Q}_2}(f)$, the Julia set $J_{\mathbb{Q}_2}(f)$ and the grand orbit $\mathrm{GO}_{\Q_2}(1)$, respectively. We begin with the determination of the Julia and Fatou sets of $f$.

\subsection{Julia and Fatou sets of $f$}\label{sec:dynamics}
First, note that the fixed point $\infty$ is in $F(f)$. Further, $\{\infty\}$ is totally $f$-invariant,  we thus need only focus on the dynamics in $\mathbb{Q}_2$. Note also that the fixed points of $f$ are $0$, $1/3$ and $5/3$, and the critical points of $f$ are $1$ and $1/3$. Moreover, both of the critical points are wild. Observe that $1/3$ and $5/3$ are contained in $F_{\C_2}(f)$, and hence in $F(f)$. The point $0$ is a repelling fixed point with multiplier $9/4$, and hence $0\in J(f)$. The critical point $1$ is mapped to $0$ by $f$. It follows that $1\in J_{\C_2}(f)\cap\mathbb{Q}_2$, and  $J_{\C_2}(f)\cap\mathbb{Q}_2$ contains no recurrent critical points. By Theorem \ref{Thm:main1}, we have $1\in J(f)$. Moreover, we can define a component of $F(f)$ to be the restriction of a component of $F_{\C_2}(f)$ to $\mathbb{Q}_2$.

To characterize the Julia dynamics, we first consider the Fatou set. Recall that the ring of integers of $\mathbb{Q}_2$ is $\mathbb{Z}_2:=\{x\in\mathbb{Q}_2: |x|\le 1\}$. For a point $a\in F(f)$, denote by $\Omega_a$ the component of $F(f)$ containing $a$. We also denote by $\Omega_\infty$ the basin of $\infty$ for $f$, that is, the set of points in $\Q_2$ whose orbits diverge to $\infty$.
\begin{lemma}\label{basin-infty}
The Fatou component $\Omega_\infty$ contains $\mathbb{Q}_2\setminus\mathbb{Z}_2$.
\end{lemma}
\begin{proof}
For $x\in\mathbb{Q}_2$ with $|x|>1$, we have
$$|f(x)|=4|x||(x-1)^2|=4|x^3|>|x|.$$
Then  $\mathbb{Q}_2\setminus\mathbb{Z}_2\subset\Omega_\infty$.
\end{proof}

Now we concentrate our attention on $\mathbb{Z}_2$.
Note that
$$\mathbb{Z}_2=2\mathbb{Z}_2\cup (1+2\mathbb{Z}_2)=4\mathbb{Z}_2\cup (2+4\mathbb{Z}_2)\cup (1+4\mathbb{Z}_2)\cup (1+2+4\mathbb{Z}_2).$$
The dynamics on the disks $2+4\mathbb{Z}_2$ and $1+2+4\mathbb{Z}_2$ are quite simple:

\begin{lemma}\label{basin-1/3}
The disks $2+4\mathbb{Z}_2$ and $1+2+4\mathbb{Z}_2$ are in $F(f)$. More precisely,
\begin{enumerate}
\item $2+4\mathbb{Z}_2\subset\Omega_\infty$; 
\item $1+2+8\mathbb{Z}_2\subset \Omega_{1/3}$; and 
\item $1+2+4+8\mathbb{Z}_2\subset \Omega_{5/3}$.
\end{enumerate}
\end{lemma}
\begin{proof}
For the statement $(1)$, pick $x\in 2+4\mathbb{Z}_2$. We have
$|f(x)|=4|x|>1.$
Hence by Lemma \ref{basin-infty}, we conclude that $f(2+4\mathbb{Z}_2)\subset \Omega_\infty$.  Therefore, $2+4\mathbb{Z}_2\subset\Omega_\infty$.

For the statement $(2)$, note that
$$\frac{1}{3}=1+2+2^3+2^5+\cdots\in1+2+4\mathbb{Z}_2.$$
Pick an arbitrary point $x\in 1+2+8\mathbb{Z}_2$. Then $3(1-x)/2\in 1+2\mathbb{Z}_2$. Writing $3(1-x)/2=1+y$ with $y\in 2+4\mathbb{Z}_2$, we have
$$\left|f(x)-\frac{1}{3}\right|=\left|\frac{9}{4}x(1-x)^2-\frac{1}{3}\right|=\left|x(\frac{3}{2}(1-x))^2-\frac{1}{3}\right|=\left|x+2xy+xy^2-\frac{1}{3}\right|\le\left|x-\frac{1}{3}\right|.$$
Hence $1+2+8\mathbb{Z}_2\subset \Omega_{1/3}$.

Applying similar argument as above, we conclude that statement $(3)$ holds.
\end{proof}

Now we analyze the orbit of the disks in $4\mathbb{Z}_2$ and $1+4\mathbb{Z}_2$, which allows us to construct a subshift on an alphabet of countably many symbols in Section \ref{sec:example-proof}. We begin with the disk $4\mathbb{Z}_2$.

\begin{proposition}\label{near 0}
The map $f$ is scaling on $4\mathbb{Z}_2$ with scaling ratio $4$. Moreover, for $n\ge 1$ and $a,b\in\{0,1\}$, the following hold:
\begin{enumerate}
\item $f(2^{n+1}\mathbb{Z}_2)=2^{n-1}\mathbb{Z}_2$;
\item $f(2^{n+1}+2^{n+2}\mathbb{Z}_2)=2^{n-1}+2^{n}\mathbb{Z}_2$;
\item $f(2^{n+1}+a2^{n+2}+2^{n+3}\mathbb{Z}_2)=2^{n-1}+a2^{n}+2^{n+1}\mathbb{Z}_2$; and
\item $f(2^{n+1}+a2^{n+2}+b2^{n+3}+2^{n+4}\mathbb{Z}_2)=2^{n-1}+a2^{n}+b2^{n+1}+2^{n+2}\mathbb{Z}_2$.
\end{enumerate}
\end{proposition}

\begin{proof}
Observe that
$$|f(x)-f(y)|=4|(x^3-2x^2+x)-(y^3-2y^2+y)|=4|x-y||x^2+y^2+xy-2(x-y)+1|.$$
Then for $x,y\in 4\mathbb{Z}_2$,
$$|f(x)-f(y)|=4|x-y|.$$
Thus any disk contained in $4\mathbb{Z}_2$ is mapped to a disk with radius multiplied by $4$.
Then the statement $(1)$ follows immediately. Direct computation shows
$$|f(2^{n+1}+a2^{n+2}+b2^{n+3})-(2^{n-1}+a2^{n}+b2^{n+1})|=2^{n+2}.$$
Hence, the statements $(2)-(4)$ hold.
\end{proof}

The following corollary gives more details of the orbits of the disks in $4\mathbb{Z}_2$.

\begin{corollary}\label{even}
For $n\ge 1$, the following hold:
\begin{enumerate}
\item $f(2^{2n}+2^{2n+3}\mathbb{Z}_2)=2^{2n-2}+2^{2n+1}\mathbb{Z}_2,$
in particular,
 $$f(2^2+2^5\mathbb{Z}_2)=\{1\}\bigcup\bigcup_{m\ge 2}((1+2^{m+1}+2^{m+3}\mathbb{Z}_2)\cup (1+2^{m+1}+2^{m+2}+2^{m+3}\mathbb{Z}_2));$$
\item $f(2^{2n}+2^{2n+2}+2^{2n+3}\mathbb{Z}_2)=2^{2n-2}+2^{2n}+2^{2n+1}\mathbb{Z}_2$;
\item $f(2^{2n+1}+2^{2n+2}\mathbb{Z}_2)=2^{2n-1}+2^{2n}\mathbb{Z}_2,$
in particular,
$$f^{n+1}(2^{2n-1}+2^{2n}\mathbb{Z}_2)\subset\mathbb{Q}_2\setminus\mathbb{Z}_2\subset\Omega_\infty;$$
 \item $f(2^{2n}+2^{2n+1}+2^{2n+2}\mathbb{Z}_2)=2^{2n-2}+2^{2n-1}+2^{2n}\mathbb{Z}_2,$
 in particular, $$f^n(2^{2n}+2^{2n+1}+2^{2n+2}\mathbb{Z}_2)=1+2+2^2\mathbb{Z}_2\subset\Omega_{1/3}\cup\Omega_{5/3}.$$
\end{enumerate}
\end{corollary}
\begin{proof}
The statements $(1), (2)$ and the first assertions of the statements $(3), (4)$ follow immediately from Proposition \ref{near 0}.  For the remaining assertions, by induction, we have
$$f^{n+1}(2^{2n+1}+2^{2n+2}\mathbb{Z}_2)=(2^{-1}+\mathbb{Z}_2)\subset\mathbb{Q}_2\setminus\mathbb{Z}_2,$$
and
$$f^n(2^{2n}+2^{2n+1}+2^{2n+2}\mathbb{Z}_2)=1+2+2^2\mathbb{Z}_2.$$
Then by Lemmas \ref{basin-infty} and \ref{basin-1/3}, the second assertions of the statements $(3)$ and $(4)$ hold.
\end{proof}

On the disk $1+4\mathbb{Z}_2$ the dynamics is more subtle, since the image of a disk is not necessary a disk. However, by Lemma \ref{lem:local scaling poly}, for any point $x_0\in 1+4\mathbb{Z}_2$ with $x_0\neq1$ and $x_0\neq1/3$,  there is a maximal scaling disk $\mathcal{S}_{x_0}(f)$ of $x_0$. Then for any disk $D\subset\mathcal{S}_{x_0}(f)$, the image $f(D)$ is also a disk, since the map $f$ is scaling on $D$ with scaling ratio $|f'(x_0)|$.

\begin{proposition}\label{near 1}
For $n\ge 1$, the following hold:
\begin{enumerate}
\item $f(1+2^n\mathbb{Z}_2)\subset 2^{2n-2}\mathbb{Z}_2;$
\item $f(1+2^{n+1}+2^{n+3}\mathbb{Z}_2)=2^{2n}+2^{2n+3}\mathbb{Z}_2$ for $n\ge 2$, and
$f(1+2^{2}+2^{4}\mathbb{Z}_2)=2^{2}+z^4+2^{5}\mathbb{Z}_2$,
in particular, $f$ is scaling on $1+2^{n+1}+2^{n+3}\mathbb{Z}_2$ with scaling ratio $2^{-n}$; and
\item  $f(1+2^{n+1}+2^{n+2}+2^{n+3}\mathbb{Z}_2)=2^{2}+2^4+2^{5}\mathbb{Z}_2$  for $n\ge 2$, and
$f(1+2^{2}+2^{4}\mathbb{Z}_2)=2^{2}+2^4+2^{5}\mathbb{Z}_2$,
in particular, $f$ is scaling on $1+2^{n+1}+2^{n+2}+2^{n+3}\mathbb{Z}_2$ with scaling ratio $2^{-n}$.
\end{enumerate}
\end{proposition}
\begin{proof}
For the statement $(1)$, pick $x\in 1+2^n\mathbb{Z}_2$. Then $|x|=1$ and $|x-1|\le 2^{-n}$. Hence
$$|f(x)|=4|x||x-1|^2\le 2^{-2n+2},$$
and the conclusion follows.

For the statement $(2)$, pick $x_0\in 1+2^{n+1}+2^{n+3}\mathbb{Z}_2$. We calculate the scaling disk $\mathcal{S}_{x_0}(f)$ of $x_0$. Consider the Taylor series of $f$ near $x_0$:
$$f(x)-f(x_0)=f'(x_0)(x-x_0)+\frac{f''(x_0)}{2}(x-x_0)^2+\frac{f''(x_0)}{6}(x-x_0)^3.$$
By Lemma \ref{lem:local scaling poly}, it suffices to pick $r_{x_0}>0$ such that
$$\left|\frac{f''(x_0)}{2}\right| r_{x_0}<|f'(x_0)|\ \ \text{and}\ \ \left|\frac{f'''(x_0)}{6}\right| r^2_{x_0}<|f'(x_0)|.$$
Direct computations show that $|f'(x_0)|=2^{-n}$, $|f''(x_0)|=2$, and $|f'''(x_0)|=2$. Hence $r_{x_0}<2^{-(n+2)}$. Therefore, $\mathcal{S}_{x_0}(f)=D(x_0, 2^{-(n+3)})$. It follows that for any $x\in\mathcal{S}_{x_0}(f)$,
$$|f(x)-f(x_0)|=|f'(x_0)||x-x_0|=2^{-n}|x-x_0|.$$
Note that
$$f(1+2^{n+1})= 2^{2n}+2^{2n+3}+2^{3n+1}+2^{3n+4}.$$
Then the statement $(2)$ holds.

For the statement $(3)$, pick $y_0\in1+2^{n+1}+2^{n+2}+2^{n+3}\mathbb{Z}_2$. Again by Lemma \ref{lem:local scaling poly}, we have $\mathcal{S}_{y_0}(f)=D(y_0, 2^{-(n+3)})$. Direct computations give us
$$|f(1+2^{n+1}+2^{n+2})-2^{2n}|=2^{-(2n+4)}.$$
Then the statement $(3)$ holds.
\end{proof}

\begin{corollary}
The Fatou components $\Omega_{1/3}=1+2+8\mathbb{Z}_2$ and $\Omega_{5/3}=1+2+4+8\mathbb{Z}_2$.
\end{corollary}
\begin{proof}
It follows immediately from Lemma \ref{basin-1/3} and Proposition \ref{near 1}.
\end{proof}

Now we can characterize the Julia set $J(f)$. Lemmas \ref{basin-infty} and \ref{basin-1/3} imply that $J(f)$ is contained in $4\mathbb{Z}_2\cup (1+4\mathbb{Z}_2)$. In fact, we have the following proposition.
\begin{proposition}\label{Julia}
The Julia set of $f$ is
$$J(f)=\bigcap_{n\ge 1}f^{-n}(4\mathbb{Z}_2\cup (1+4\mathbb{Z}_2)).$$
\end{proposition}
\begin{proof}
If $x\not\in\bigcap_{n\ge 1}f^{-n}(4\mathbb{Z}_2\cup (1+4\mathbb{Z}_2))$, then there exists $n_0\ge 0$ such that
$$f^{n_0}(x)\in(\mathbb{Q}_2\setminus\mathbb{Z}_2)\cup (2+4\mathbb{Z}_2)\cup (1+2+4\mathbb{Z}_2).$$
By Lemmas \ref{basin-infty} and \ref{basin-1/3}, we conclude that $x\in F(f)$. Hence 
$J(f)\subset\cap_{n\ge 1}f^{-n}(4\mathbb{Z}_2\cup 1+4\mathbb{Z}_2).$

Conversely, pick $y\in\bigcap_{n\ge 1}f^{-n}(4\mathbb{Z}_2\cup (1+4\mathbb{Z}_2))$. For any disk $D:=D_{\mathbb{Q}_2}(y, r)$, we claim that $\{f^k\}_{k\ge 1}$ is not equicontinuous on $D$. To ease notations, set
$$W_0:=\bigcup_{n\ge 1}(1+2^{n+1}+2^{n+3}\mathbb{Z}_2),\ \ \text{and}\ \ W_1:=\bigcup_{n\ge 1}(1+2^{n+1}+2^{n+2}+2^{n+3}\mathbb{Z}_2).$$
 Shrinking $D$ if necessary, let $k_0\ge 0$ be the smallest integer such that $f^{k_0}(D)\subset W_0$ or $f^{k_0}(D)\subset W_1$. If there is no such $k_0$, then by Proposition \ref{near 0} and Corollary \ref{even}, we have $y\in\mathrm{GO}_{\Q_2}(1)$, which implies $y\in J_{\mathbb{C}_2}(f)$, and hence by Theorem \ref{Thm:main1}, $y\in J(f)$. Now, we inductively define $k_{i}>k_{i-1}$ to be the smallest positive integer such that $f^{k_i}(D)\subset W_0$ or $f^{k_i}(D)\subset W_1$. By Propositions \ref{near 0} and \ref{near 1}, there exists $m_i\ge 1$ such that the scaling ratio of $f$ on $f^{k_i}(D)$ is $2^{-m_i}$, and hence the scaling ratio of $f^{k_{i+1}-k_i-1}$ on $f^{k_i+1}(D)$ is $4^{m_i}$. In fact, the relation between $k_j$ and $m_i$ is $m_i=k_{i+1}-k_i-1$. Hence
 $\mathrm{diam}(f^{k_i}(D))=2^{m_1+m_2\cdots+m_i}\cdot 4^{n_0}r.$
 Since $m_i\ge 1$, there are only finitely many such $k_i$. Again by Propositions \ref{near 0} and \ref{near 1}, there exists $\ell>\max\{k_i: i\ge 0\}$ such that $f^{\ell}(D)\cap (1+4\mathbb{Z}_2)\not=\emptyset$. Then $f^{\ell}(D)$ contains $1$. Hence $D$ contains an $\ell$-th pre-image of $1$. Since $1\in J(f)$, by Lemma \ref{regular-point-lemma},
 $D\cap J(f)\not=\emptyset.$
 Thus the claim holds, and hence $y\in J(f)$.
\end{proof}

\begin{corollary}\label{Fatou}
The Fatou set of $f$ is
$$F(f)=\bigcup_{n\ge 0}f^{-n}(\Omega_\infty\cup\Omega_{1/3}\cup\Omega_{5/3}).$$
In particular, any component of $F(f)$ is eventually mapped to $\Omega_\infty$, $\Omega_{1/3}$ or $\Omega_{5/3}$.
\end{corollary}
\begin{proof}
Pick $x\in F(f)$ and set $\Omega=\Omega_\infty\cup\Omega_{1/3}\cup\Omega_{5/3}$. Suppose that $x\not\in\bigcup_{n\ge 0}f^{-n}(\Omega)$. Then by Lemmas \ref{basin-infty} and \ref{basin-1/3}, for any $n\ge 0$, we have $f^n(x)\in 4\mathbb{Z}_2\cup (1+4\mathbb{Z}_2)$. Proposition \ref{Julia} implies that $x\in J(f)$, which is a contradiction. Hence $x\in\bigcup_{n\ge 0}f^{-n}(\Omega)$.

For the other direction, if $x\in\bigcup_{n\ge 0}f^{-n}(\Omega)$, then there exists $n_0\ge 0$ such that $f^{n_0}(x)\in\Omega$. We claim that $x\in F(f)$. If $x\in J(f)$, then from the portrait of the critical points of $f$ we know that $f^{i}(x)$ is not a critical point for $0\le i\le n_0-1$. Hence, by Lemma \ref{regular-point-lemma}, $f^{n_0}(x)\in J(f)$, which is impossible since $f^{n_0}(x)\in\Omega\subset F(f)$.
\end{proof}

\subsection{The tree of closed disks}\label{sec:tree}
To illustrate Corollary \ref{even} and Proposition \ref{near 1}, we use the tree of closed disks. Let $\mathcal{T}$ be the space of closed disks in $\mathbb{Q}_2$ with radii in the value group $|\mathbb{Q}^\times_2|$. Since for any two closed disks in $\mathbb{Q}_2$, either they are disjoint or one contains the other, the inclusion gives an order relation on $\mathcal{T}$. We assign an undirected edge $[A,B]$ for the two elements $A$ and $B$, if $B$ is the smallest element in $\mathcal{T}$ that is larger than $A$. Then the space $\mathcal{T}$ admits a tree structure. The ends of the tree $\mathcal{T}$ are points in $\mathbb{P}^1_{\mathbb{Q}_2}$. Denote by $\partial\mathcal{T}$ the set of the ends and define $\overline{\mathcal{T}}:=\mathcal{T}\cup\partial{\mathcal{T}}$. For such trees in a more general setting, we refer to \cite[Section 1.2]{Favre18}.

Since $\mathbb{Q}_2$ is not algebraically closed, the polynomial $f$ does not induce a well-defined map on $\overline{\mathcal{T}}$. However, after removing countably many elements in $\overline{\mathcal{T}}$, there is a well-defined map on the corresponding complement. By Lemma \ref{basin-infty}, we now focus on the closed disks in $\mathbb{Z}_2$. Let $\mathcal{T}_0\subset\mathcal{T}$ be the set of closed disks in $2\mathbb{Z}_2$ and let $\mathcal{T}_1\subset\mathcal{T}$ be the set of closed disks in $1+2\mathbb{Z}_2$ removing the ones containing the critical point $1$ or $1/3$, that is, we remove the disks of forms $1+2^n\mathbb{Z}_2$, $1+\sum_{k=1}^n2^{2k-1}+2^{2n}\mathbb{Z}_2$ and $1+\sum_{k=1}^n2^{2k-1}+2^{2n+1}\mathbb{Z}_2$. Denote by $\partial\mathcal{T}_0$ and $\partial\mathcal{T}_1$ the sets of ends for $\mathcal{T}_0$ and $\mathcal{T}_1$, respectively. Let $\overline{\mathcal{T}}_0:=\mathcal{T}_0\cup\partial{\mathcal{T}}_0$ and $\overline{\mathcal{T}}_1:=\mathcal{T}_1\cup\partial{\mathcal{T}}_1$. Then the polynomial $f$ induces a well-defined map
$$F:\overline{\mathcal{T}}_0\cup\overline{\mathcal{T}}_1\to\overline{\mathcal{T}}.$$

To view Corollary \ref{even} and Proposition \ref{near 1} from the map $F$, for $n\ge 1$, we define
\begin{equation*}
\begin{split}
& A_n:=2^{2n}+2^{2n+3}\mathbb{Z}_2, \\
 &A'_n:=2^{2n}+2^{2n+2}+2^{2n+3}\mathbb{Z}_2,\\
 &B_n:=1+2^{n+1}+2^{n+3}\mathbb{Z}_2,\\
 &B'_n:=1+2^{n+1}+2^{n+2}+2^{n+3}\mathbb{Z}_2.
 \end{split}
\end{equation*}
Then under $F$, we have the following portrait (see Figure \ref{figure:1}):

\begin{figure}[h!]
\centering
\includegraphics[width=0.9\textwidth]{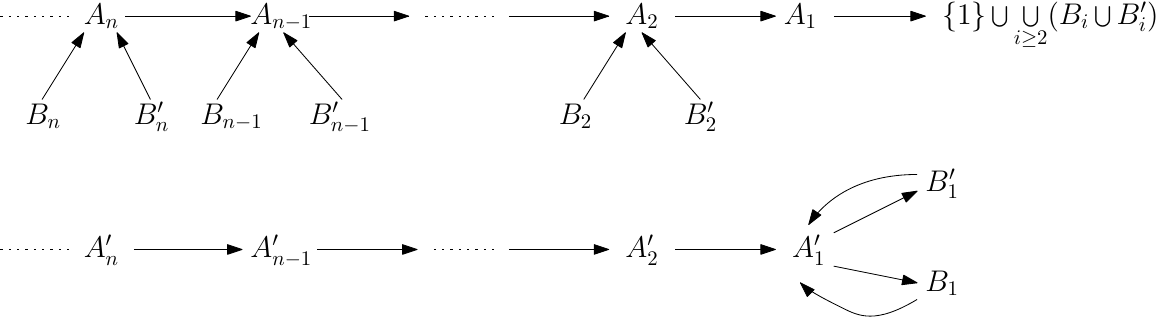}
\caption{The portrait of $A_n$, $A'_n$, $B_n$ and $B'_n$ under $F$.}
\label{figure:1}
\end{figure}

\subsection{Proof of Theorem \ref{Thm:example}}\label{sec:example-proof}
We first establish several lemmas to relate the Julia dynamics of $f$ to the symbolic dynamics. Recall the notations from Section \ref{sec:tree}. Set
$$\begin{cases}
\alpha_n:=A_n\cap J(f),\\
\alpha'_n:=A'_n\cap J(f),\\
\beta_n:=B_n\cap J(f),\\
\beta'_n:=B'_n\cap J(f),\\
\alpha_\infty:=\{0\},\\
\beta_\infty:=\{1\}.
\end{cases}$$
Let
$$\mathcal{A}:=\{\alpha_\infty, \beta_\infty, \alpha_1,\alpha'_1, \beta_1, \beta'_1, \cdots, \alpha_n,\alpha'_n, \beta_n, \beta'_n,\cdots\}.$$
Define the matrix $A=(A_{\gamma_i,\gamma_j})_{\gamma_i,\gamma_j\in\mathcal{A}}$ by
$$A_{\gamma_i,\gamma_j}=
\begin{cases}
1\ \ \text{if}\  \ \gamma_j\subset f(\gamma_i),\\
0\ \ \text{otherwise}.
\end{cases}$$
 From Section \ref{sec:tree}, we know that
 \begin{enumerate}
 \item if $\gamma_i\not=\alpha_1$, $\alpha'_1$, then there is a unique $\gamma_j\in\mathcal{A}$ such that $A_{\gamma_i,\gamma_j}=1$;
 \item if $\gamma_i=\alpha_1$, we have $A_{\alpha_1,\gamma_j}=1$ if and only if $\gamma_j=\beta_\infty, \beta_n$ or $\beta'_n$ for some $n\ge 2$; and
 \item if $\gamma_i=\alpha'_1$, then $A_{\gamma_i,\gamma_j}=1$ if and only if $\gamma_j\in\{\beta_1, \beta'_1\}$.
 \end{enumerate}

Let $(\Sigma_A,\sigma)$ be the corresponding subshift on the alphabet $\mathcal{A}$ defined by $A$. From Section \ref{sec:dynamics}, we find that the set $\mathcal{A}$ is a partition of $J(f)$. Thus we immediately obtain an element in $\Sigma_A$ for each point $x\in J(f)$:
\begin{lemma}\label{J-to-sequence}
For each $x\in J(f)$, there is a unique sequence $(\gamma_j)_{j\ge 0}\in\Sigma_A$ such that $f^j(x)\in\gamma_j$.
\end{lemma}

We say that the sequence in Lemma \ref{J-to-sequence} is the \textit{code sequence} of $x$. Consider the coding map
$$h: J(f)\to\Sigma_A,$$
sending $x$ to its code sequence. The next result claims that the map $h$ is a bijection.

\begin{lemma}\label{bijection}
For any $(\gamma_j)_{j\ge 0}\in\Sigma_A$, there is a unique point $x\in J(f)$ such that the code sequence of $x$ is $\{\gamma_j\}_{j\ge 0}$.
\end{lemma}

To prove Lemma \ref{bijection}, we first define the sequence $\{T_{\gamma_j,\gamma_{j+1}}\}_{j\ge 0}$ of inverse maps along $\{\gamma_j\}_{j\ge 0}$:
$$T_{\gamma_j\gamma_{j+1}}: \gamma_{j+1}\to \gamma_{j}$$
satisfying $f\circ T_{\gamma_j\gamma_{j+1}}=id_{\gamma_{j+1}}$. By Propositions \ref{near 0} and \ref{near 1}, if $\gamma_{j}\in\{\alpha_1,\alpha'_1,\cdots\}$, the inverse map $T_{\gamma_j,\gamma_{j+1}}$ is scaling with scaling ratio $1/4$; and if $\gamma_{j}\in\{\beta_n,\beta'_n\}$, the map $T_{\gamma_j,\gamma_{j+1}}$ is scaling with scaling ratio $2^n$.

Now we can prove Lemma \ref{bijection}.

\begin{proof}[Proof of Lemma \ref{bijection}]
Pick a sequence $\{y_j\}_{j\ge 0}\subset J(f)$ such that $y_j\in\gamma_j$.  Define
$$x_j:=T_{\gamma_0\gamma_1}\circ\cdots\circ T_{\gamma_{j-1}\gamma_j}(y_j).$$
If there exists $j_0\ge 0$ such that $\gamma_{j_0}=\alpha_\infty$, then $\gamma_j=\alpha_\infty$ for all $j\ge j_0$. Hence in this case, $x_j$ converges. Now we assume that $\gamma_{j}\not=\alpha_\infty$ for all $j\ge 0$.
Let $0\le i_0<i_1<\cdots <i_{k(j)}\le j-1$
be all integers such that $\gamma_{i_\ell}\subset 1+4\mathbb{Z}_2$ for $1\le\ell\le k(j)$. Note that $k(j)\to\infty$, as $j\to\infty$. To ease notations, set $k=k(j)$. Rewrite
$$x_j=T_{\gamma_0\gamma_1}\circ\cdots\circ T_{\gamma_{i_0}\gamma_{i_0+1}}\circ\cdots\circ T_{\gamma_{i_k}\gamma_{i_k+1}}\circ\cdots\circ T_{\gamma_{j-1}\gamma_j}(y_j).$$
By Propositions \ref{near 0} and \ref{near 1}, for any $0\le\ell\le k-1$, the map
$T_{\gamma_{i_\ell}\gamma_{i_\ell+1}}\circ\cdots\circ T_{\gamma_{i_{\ell+1}-1}\gamma_{i_{\ell+1}}}$
is  scaling with scaling ratio no more than $1/2$. Moreover, for $0\le i\le i_0-1$, the map $T_{\gamma_i,\gamma_{i+1}}$ is scaling with scaling ratio $1/4$. Hence
$$\left|x_{j+m}-x_j\right|\le \frac{1}{2^k}\left|T_{\gamma_k\gamma_{k+1}}\circ\cdots\circ T_{\gamma_{j+m-1}\gamma_{j+m}}(y_{j+m})-T_{\gamma_k\gamma_{k+1}}\circ\cdots\circ T_{\gamma_{j-1}\gamma_j}(y_j)\right|.$$
By Proposition \ref{Julia}, we have $\gamma_k\subset J(f)\subset\mathbb{Z}_2$. It follows that
$$\left|x_{j+m}-x_j\right|\le \frac{1}{2^k}.$$
Hence $\{x_j\}_{j\ge 0}$ is a Cauchy sequence.  Set
$$x:=\lim\limits_{j\to\infty}x_j.$$
Since the scaling ratio of $T_{\gamma_i,\gamma_{i+1}}$ is in the value group $|\mathbb{Q}^\times_2|$, all $x_j$ are in $\mathbb{Q}_2$. Therefore, the point $x$ is in $\mathbb{Q}_2$.

Observe that $x$ is independent of the choice of $\{y_j\}_{j\ge 0}$. Indeed, for any other sequence $\{y'_j\}_{j\ge 0}\subset J(f)$ with $y'_j\in\gamma_j$, let
$$x'_j:=T_{\gamma_0\gamma_1}\circ\cdots\circ T_{\gamma_{j-1}\gamma_j}(y'_j).$$
Then
$$\left|x'_j-x_j\right|\le \frac{1}{2^k}\left|T_{\gamma_k\gamma_{k+1}}\circ\cdots\circ T_{\gamma_{j-1}\gamma_j}(y_j)-T_{\gamma_k\gamma_{k+1}}\circ\cdots\circ T_{\gamma_{j-1}\gamma_j}(y'_j)\right|$$
Again, noting that $\gamma_k\subset J(f)\subset\mathbb{Z}_2$, we have $|x'_j-x_j|\to 0$, as $j\to\infty$.

Obviously, $x\in J(f)$ and the code sequence of $x$ is $\{\gamma_j\}_{j\ge 0}$. Now we show that such $x$ is unique. We may assume that $\gamma_{j}\not=\alpha_\infty$ for all $j\ge 0$. Otherwise, the uniqueness of $x$ follows immediately from the fact that $f$ is a bijection on each $\gamma\in\mathcal{A}$. Suppose that $\tilde{x}\in J(f)$ also has the code sequence $\{\gamma_j\}_{j\ge 0}$. Then $f^n(x)$ and $f^n(\tilde{x})$ are in $\gamma_n$. Hence $|f^n(x)-f^n(\tilde{x})|$ is bounded for all $n\ge 0$. However, there exists a sequence $\{n_i\}$ such that $f^{n_i}$ is scaling on $\gamma_0$ with scaling ratio no less than $2^i$. Hence
$$|f^{n_i}(x)-f^{n_i}(\tilde{x})|\ge 2^{i}|x-\tilde{x}|.$$
Thus $\tilde{x}=x$.
\end{proof}

Recall that $\sigma:\Sigma_A\to\Sigma_A$ is the (left-)shift.
Considering the code sequences in Lemma \ref{bijection}, we obtain a conjugacy between $f$ and $\sigma$ via $h$.
\begin{proposition}\label{comjugacy}
On $J(f)$,
$$h\circ f=\sigma\circ h.$$
\end{proposition}
\begin{proof}
For $x\in J(f)$, let $\{\gamma_j\}_{j\ge 0}$ be its code sequence. Then for any $\{y_j\}_{j\ge 0}\subset J(f)$ with $y_j\in\gamma_j$, we have
$$h\circ f(x)=h\circ f\left(\lim\limits_{j\to\infty}T_{\gamma_0\gamma_1}\circ\cdots\circ T_{\gamma_{j-1}\gamma_j}(y_j)\right).$$
Note that $f\circ T_{\gamma_0\gamma_1}=id_{\gamma_1}$. It follows that
$$h\circ f(x)=h \left(\lim\limits_{j\to\infty}T_{\gamma_1\gamma_2}\circ\cdots\circ T_{\gamma_{j-1}\gamma_j}(y_j)\right)=(\gamma_1,\gamma_2,\cdots)=\sigma\circ h(x).$$
The conclusion follows.
\end{proof}

Now we are ready to prove Theorem \ref{Thm:example}.
\begin{proof}
By Propositions \ref{bijection} and  \ref{comjugacy}, it suffices to prove that the map $h$ is continuous and open from $I(f)$ to $\Sigma_A\setminus h(\mathrm{GO}(1))$.  We first show the continuity of $h$. Let $U$ be an open set in $h(I(f))$ and pick $x\in I(f)$ with $h(x)\in U$. Write $h(x)=(\gamma_i)_{i\ge 0}$. Then there exists $m\ge 1$ such that the cylinder set $C:=[\gamma_0\cdots\gamma_m]\subset\Sigma_A$
is contained in $U$. For the set $h^{-1}(C)$, we have $x\in h^{-1}(C)$ and $h^{-1}(C)=\bigcap_{j=0}^{m}f^{-j}(\gamma_j)$.
Since  $x\not\in\mathrm{GO}(1)$, we have $\gamma_j\not\in\{\alpha_\infty$, $\beta_\infty\}$ and hence $\gamma_j$ is an open set in $J(f)$. Since $f$ is continuous, then $f^{-j}(\gamma_j)$ is open and hence $h^{-1}(C)$ is open. Therefore, $h$ is continuous.

Now we show the openness of $h$. Let $W$ be an open subset in $I(f)$. By Propositions \ref{near 0} and \ref{near 1}, there exists $\ell\ge 0$ such that
$f^\ell(W)\supset (1+4\mathbb{Z}_2)\cap I(f).$
Then
$$f^{\ell+1}(W)\supset 4\mathbb{Z}_2\cap I(f),\ \text {and hence}\ f^{\ell+2}(W)\supset \mathbb{Z}_2\cap I(f)=I(f).$$
Pick $(\gamma_i)_{i\ge 0}\in h(W)$. Since $f$ is a bijection from $\gamma$ to $f(\gamma)$ for any $\gamma\in\mathcal{A}$, the set $h(W)$ contains the cylinder $[\gamma_0\cdots\gamma_{\ell+1}]\subset\Sigma_A$.
Therefore, $h$ is open.
\end{proof}

\subsection{Gurevich entropy}\label{sec:Gurevich}

Consider the natural directed graph $\Gamma_A$ induced by the matrix $A$ in Section \ref{sec:example-proof}, that is, $\Gamma_A$ has vertices in $\mathcal{A}$, and possesses an edge from $\gamma_i\in\mathcal{A}$ to $\gamma_j\in\mathcal{A}$ if and only if $A_{\gamma_i\gamma_j}=1$.  From Section \ref{sec:tree}, we have that the graph $\Gamma_A$ is not strongly connected. Hence the matrix $A$ is reducible. Now we consider its irreducible component that corresponds to a maximal strongly connected subgraph of $\Gamma_A$. In fact, the matrix $A$ has an irreducible component $A'$ corresponding to the symbol set
$$\mathcal{A}':=\{\alpha_1,\alpha_2,\beta_2,\beta'_2,\cdots,\alpha_n,\beta_n,\beta'_n,\cdots\}.$$
Denote by $\Sigma'_{A'}$ the corresponding  subspace and let $(\Sigma'_{A'},\sigma')$ be the corresponding subsystem of $(\Sigma_A,\sigma)$, where $\sigma'=\sigma|_{\Sigma'_{A'}}$. We now compute the Gurevich entropy $h_G(\Sigma'_{A'}, \sigma')$ of the system $(\Sigma'_{A'},\sigma')$.

Note that the matrix $A'$ has the form
\[
\begin{blockarray}{&ccccccccccccccc}
 & \alpha_1 & \alpha_2 & \beta_2&\beta'_2&\alpha_3 & \beta_3&\beta'_3 & \alpha_4 & \beta_4&\beta'_4&\alpha_5& \beta_5&\beta'_5&\dots\\
\begin{block}{c(ccccccccccccccc)}
 \alpha_1& & &1&1&&1&1&&1&1&&1&1&\dots\\
 \alpha_2&1&&&&&&&&&&&&&\dots\\
 \beta_2&&1&&&&&&&&&&&&\dots\\\
 \beta'_2&&1&&&&&&&&&&&&\dots\\
 \alpha_3&&1&&&&&&&&&&&&\dots\\
  \beta_3&&&&&1&&&&&&&&&\dots\\
 \beta'_3&&&&&1&&&&&&&&&\dots\\
  \alpha_4& &&&&1&&&&&&&&&\dots\\
  \beta_4 &  &&&&&&&1&&&&&&\dots\\
 \beta'_4 &  &&&&&&&1&&&&&&\dots\\
\alpha_5&  &&&&&&&1&&&&&&\dots\\
 \beta_5 &   &&&&&&&&&&1&&&\dots\\
\beta'_5 &  &&&&&&&&&&1&&&\dots\\
\vdots & \vdots &\vdots&\vdots&\vdots&\vdots&\vdots&\vdots&\vdots&\vdots&\vdots&\vdots&\vdots&\vdots&\ddots\\
\end{block}
\end{blockarray}\ .
 \]
Then the number of first return loops at $\alpha_1$ of length $n\ge 1$ is
$$\Xi_{\alpha_1}(n)=
\begin{cases}
0\ \ \text{if}\ \ n\le 2,\\
2 \ \ \text{if}\ \ n\ge 3.
\end{cases}$$
Recall Section \ref{sec:symbols} and consider
$$G_{\alpha_1}(z)=\sum_{n\ge 1}\Xi_{\alpha_1}(n)z^n=\sum_{n\ge 3}2z^n.$$
If $|z|<1$, then
$$G_{\alpha_1}(z)=\frac{2z^3}{1-z}.$$
In this case, $1-G_{\alpha_1}(z)$ has a root $R\approx 0.58975$ that is a root of $2z^3+z-1$ and hence is algebraic. Note that $0<R<1$. Then $G_{\alpha_1}(z)$ converges on $|z|<R$. Thus,
$$h_G(\Sigma'_{A'}, \sigma')=-\log R\approx 0.52806.$$

\section{Dynamics of sub-hyperbolic maps}\label{sec:GF-maps}

The goal of this section is to prove Theorems \ref{Thm:GF} and \ref{globaldynamics}. For Theorems \ref{Thm:GF}, we first investigate the case that the Julia critical points are iteratedly prefixed in Section \ref{sec:reduced}, and then the general case in Section  \ref{sec:general}. After that, we show Theorem \ref{globaldynamics} in Section \ref{sec:pf-global}.

\subsection{Reduced version of Theorem \ref{Thm:GF}}\label{sec:reduced}

A point $x\in\mathbb{P}^1_K$ is {\it iteratedly prefixed} under a rational map $\phi\in K(z)$ if there exists $n\ge 0$ such that $\phi^n(x)$ is a fixed point of $\phi$.
\begin{theorem}\label{Thm:prefixed}
Let $\phi\in  K(z)$ be a rational map of degree at least $2$. Suppose all the critical points in $\mathrm{Crit}_K^\ast(\phi)$ are iteratedly prefixed under $\phi$. Then the conclusion of Theorem \ref{Thm:GF} holds.
\end{theorem}

\begin{remark}\label{map-to-repelling}
Since all the fixed points in Julia set are repelling, any point in $\mathrm{Crit}_K^\ast(\phi)$ in Theorem \ref{Thm:prefixed} is mapped to some repelling fixed point.
\end{remark}

The remainder of this section is devoted to proving Theorem \ref{Thm:prefixed}. To unify the proof, we consider a suitable finite extension $\widetilde{K}$ of $K$ such that $\phi$, regarded as a map acting on $\mathbb{P}^1_{\widetilde{K}}$, has a nonempty Fatou set. The existence of such extension follows form the fact that $\phi$ has a nonrepelling fixed point in $F_{\mathbb{C}_p}(\phi)$, see \cite[Proposition 4.2]{Benedetto19}. Since the conjugacy via an automorphism of $\mathbb{P}^1_{\widetilde{K}}$ preserves the corresponding dynamics, we may assume that $J_{\widetilde{K}}(\phi)$ is contained in the ring $\mathcal{O}_{\widetilde{K}}$ of integers of $\widetilde{K}$. It follows that $J_K(\phi)\subset J_{\widetilde{K}}(\phi)\subset\mathcal{O}_{\widetilde{K}.}$

By Corollary \ref{coro:tame-wild}, we can construct a {\it scaling covering} $\mathcal{P}_0$ of $J_K(\phi)$ in $\mathbb{P}^1_{\widetilde{K}}$, that is, $\mathcal{P}_0$ is a set of disks in $\mathbb{P}^1_{\widetilde{K}}$ such that
\begin{enumerate}
\item each $U\in \mathcal{P}_0$ is either the restriction of a scaling disk of $\phi$ to $\mathbb{P}^1_{\widetilde{K}}$ or a singleton,
\item $U\cap J_K(\phi)\not=\emptyset$, and
\item $J_K(\phi)\subset \bigcup\limits_{U\in\mathcal{P}_0}U$.
\end{enumerate}
Indeed, there are only finitely many maximal scaling disks intersecting $J_K(\phi)$ outside a small neighborhood $W$ of $\mathrm{Crit}^\ast_K(\phi)$, and in $W$ there are countably many maximal scaling disks intersecting $J_K(\phi)$ together with the singletons of points in $\mathrm{Crit}^\ast_K(\phi)$.

By Remark \ref{map-to-repelling}, for any $c\in\mathrm{Crit}^\ast_K(\phi)$, considering the forward orbit $\mathcal{O}_\phi(c)$ and the corresponding forward orbit of the elements in $\mathcal{P}_0$ near $c$, we can decompose a small neighborhood of each point $x\in\mathcal{O}_\phi(c)$ into  countably many scaling disks together with the singleton $\{x\}$. Then we obtain a new scaling covering from $\mathcal{P}_0$, and to abuse notation, we also denote it by $\mathcal{P}_0$.  Note that $\mathcal{P}_0$ contains countably many elements, in particular, it contains all the singletons with one element in $\bigcup_{n\ge0}\phi^n(\mathrm{Crit}^\ast_K(\phi))$. Moreover, since by Theorem \ref{no-Wandering-thm} and the fact that $F_{\C_p}(\phi)\cap \mathbb{P}^1_{K}\subset F_K(\phi)$, there are finitely many nonrepelling cycles in $\mathbb{P}^1_{K}$, we can assume that each element in $\mathcal{P}_0$ is uniformly away from the nonrepelling cycles. Furthermore, we can also assume that each element in $\mathcal{P}_0$ is uniformly away from $\mathrm{Crit}_{\C_p}(\phi)\setminus\mathrm{Crit}^\ast_K(\phi)$.

The next result claims that any nonsingleton element in $\mathcal{P}_0$ has an iterated image containing at least two distinct elements in $\mathcal{P}_0$.

\begin{lemma}\label{C-cover-two}
For any $U\in\mathcal{P}_0$, suppose $U$ is not a singleton. Then there exists $m\ge 1$ such that $\phi^m(U)$ contains at least two distinct elements in $\mathcal{P}_0$.
\end{lemma}
\begin{proof}
Let $U_{\C_p}\subset\mathbb{P}^1_{\C_p}$ be the set such that $U_{\C_p}\cap\mathbb{P}^1_{\widetilde{K}}=U$.
If there exists $n_0\ge 1$ such that $\phi^{n_0}(U)\cap\mathrm{Crit}^\ast_K(\phi)\not=\emptyset$, consider the smallest such $n_0$ and pick a point $c_0$ in this intersection. It follows that $\{c_0\}$ is a singleton in $\mathcal{P}_0$. Moreover, there exists $x\in\phi^{n_0}(U)\setminus\{c_0\}$, since $U$ is not a singleton and $\phi^{n_0}$ is scaling on $U_{\C_p}$. Then letting $U_x$ be the element in $\mathcal{P}_0$ such that $x\in U_x$, we have $U_x\subset\phi^{n_0}(U)$. Hence $U_x\cup\{c_0\}\subset\phi^{n_0}(U)$, which, in this case, implies the conclusion  by taking $m=n_0$.

Now we assume that $\phi^n(U)\cap\mathrm{Crit}^\ast_K(\phi)=\emptyset$ for any $n\ge 1$.
Suppose to the contrary that there is no such $m$. Then from the construction of $\mathcal{P}_0$, the map $\phi^n$ is scaling on $U_{\C_p}$ for all $n\ge 1$. If $\mathrm{Crit}^\ast_K(\phi)=\emptyset$, then by Corollary \ref{coro:tame-wild}, {$\mathcal{P}_0$} contains finitely many elements and hence the union $\cup_{n\geq 1}\phi^n(U_{\C_p})$ disjoints $\mathrm{Crit}_{\C_p}(\phi)$. If $\mathrm{Crit}^\ast_K(\phi)\not=\emptyset$, picking $c\in\mathrm{Crit}^\ast_K(\phi)$, we have that the union $\cup_{n\geq 1}\phi^n(U_{\C_p})$ does not contain $c$. Noting that $\mathrm{Crit}_{\C_p}(\phi)$ contains at least two distinct points and $c\in J_K(\phi)\subset J_{\C_p}(\phi)$, by \cite[Theorem 5.19]{Benedetto19}, in both cases, we conclude that $U_{\C_p}\subset F_{\C_p}(\phi)$. This is impossible. Indeed, since $U\cap J_K(\phi)\not=\emptyset$, we have $U_{\C_p}$ intersects $J_{\C_p}(\phi)$.
\end{proof}

We say that a scaling covering $\mathcal{P}$ of $J_K(\phi)$ is \textit{compatible} if for each  $V\in\mathcal{P}$, the image $\phi(V)$ is a union of elements in $\mathcal{P}$.
In the following result, decomposing the elements in $\mathcal{P}_0$, we obtain a compatible scaling covering $\mathcal{P}$ of $J_K(\phi)$ in $\mathbb{P}^1_{\widetilde{K}}$, which will turn out that the intersections of elements in $\mathcal{P}$ and $J_K(\phi)$ give us the states of the desired Markov shift.
\begin{lemma}\label{lem:compatible-cover}
There is a compatible scaling covering $\mathcal{P}$ of $J_K(\phi)$ in $\mathbb{P}^1_{\widetilde K}$ such that $\mathcal{P}$ has countably many elements and each singleton in $\mathcal{P}$ contains only one element in the union $\bigcup_{n\ge0}\phi^n(\mathrm{Crit}^\ast_K(\phi))$.
\end{lemma}
\begin{proof}

Let $c\in\mathrm{Crit}^\ast_K(\phi)$.
By Remark \ref{map-to-repelling}, in the forward orbit of $c$, there is a repelling fixed point $z_c$. Consider a small neighborhood $V_{z_c}$ of $z_c$ and let $\mathcal{X}_c$ be the set of corresponding iterated preimages of $V_{z_c}$ along the orbit $\mathcal{O}_\phi(c)=\{c, \phi(c),\cdots, z_c\}$. Then define $X_c$ to be the union of all elements in $\mathcal{X}_c$, and denote  $Y:=\cup_{c\in\mathrm{Crit}^\ast_K(\phi)}X_c$. By Proposition \ref{prop:tame-wild}, the map $\phi$ has a constant scaling ratio on $V_{z_c}$.
Moreover, there is a finite subset $\mathcal{D}_{z_c}\subset \mathcal{P}_0$, each of whose elements is a subset of $V_{z_c}$, such that for any element $V\in\mathcal{P}_0$ with $V\subset V_{z_c}$, there exists a smallest $m\ge 0$ such that $\phi^m(V)\in\mathcal{D}_{z_c}$. Then for any points in $\mathcal{O}_\phi(c)$, we denote by $\mathcal{V}_c$ the set of the finitely many elements in $\mathcal{P}_0$ that are the corresponding iterated preimages of elements in $\mathcal{D}_{z_c}$ along $\mathcal{O}_\phi(c)$, and write $\mathcal{V}=\cup_{c\in\mathrm{Crit}^\ast_K(\phi)}\mathcal{V}_c$.

For a nonsingleton element $U\in\mathcal{P}_0$, let us consider the orbit of $U$ and decompose the elements in $\mathcal{P}_0$ intersecting this orbit as follows. In the case that $U$ is either in $\mathcal{V}$ or not a subset of $Y$,
if there exist $i\ge 1$ and $U_i\in\mathcal{P}_0$ such that $\phi^i(U)\subsetneq U_i$, then let $r_i:=\mathrm{diam}(\phi^i(U))$ and decompose $U_i$ into small disks in $U_i$ with diameter $r_i$. We conclude that $U_i\not=U$. Indeed, if $U_i=U$, on one hand, since $U_i\cap J_K(\phi)\not=\emptyset$, there exists $1\le j\le i-1$ such that $\phi^j(U)$ contains a critical point  $c_0\in\mathrm{Crit}^\ast_K(\phi)$ and, in turn $\phi^{i-j}(\phi^j(U))=\phi^i(U)\subsetneq U_i=U$ contains an iterated image $x$ of $c_0$, which implies $U=\{x\}$ since $U$ and $\{x\}$ are contained in $\mathcal{P}_0$; on the other hand, by the assumption, $U$ is not a singleton, which leads to a contradiction. By Lemma \ref{C-cover-two}  and noting the finiteness of such $U$ in this case, we have that the decomposition procedure stops in finitely many steps.
Now for each $U\in\mathcal{P}_0$ that is not in $\mathcal{V}$ but a subset of $Y$, there exists $\ell\ge 1$ such that $\phi^\ell(U)$ is in $\mathcal{V}$, and then we can decompose $U$ according to the decomposition on $\phi^\ell(U)$.

After decomposing, we consider the disks intersecting $J_K(\phi)$ and obtain a scaling covering $\mathcal{P}$ of $J_K(\phi)$, each of whose countably many elements is either a disk or a singleton having element in $\bigcup_{n\ge0}\phi^n(\mathrm{Crit}^\ast_K(\phi))$.

Let us check that $\mathcal{P}$ is compatible. It suffices to show that if $\phi(U')\subset U''$ for some $U', U''\in\mathcal{P}$, then $\phi(U')=U''$.  Considering iterated image if necessary, we can assume that $U'$ is either in $\mathcal{V}$ or not a subset of $Y$. Let $\widetilde U'$ and $\widetilde U''$ be the elements in $\mathcal{P}_0$ such that $U'\subset\widetilde U'$ and $U''\subset\widetilde U''$, respectively. If $U'=\widetilde U'$, from the construction of {$\mathcal{P}$}, we immediately have $\phi(U')=U''$. If $U'\subsetneq\widetilde U'$,  again from the construction of $\mathcal{P}$, there exists $U'_1\subset\widetilde U'$ such that $\mathrm{diam}(U'_1)=\mathrm{diam}(U')$ and $\mathrm{diam}(\phi(U_1'))=\mathrm{diam}(U'')$. Letting $\widetilde U'_{\C_p}\subset\mathbb{P}^1_{\C_p}$ be the set such that $\widetilde U'_{\C_p}\cap\mathbb{P}^1_{\widetilde{K}}=U'$, since $\phi$ is scaling on $\widetilde U'_{\C_p}$, we conclude that $\mathrm{diam}(\phi(U'_1))=\mathrm{diam}(\phi(U'))$ and hence $\mathrm{diam}(\phi(U'))=\mathrm{diam}(U'')$, which implies $\phi(U')=U''$.

Thus $\mathcal{P}$ is a  desired  scaling covering of $J_K(\phi)$.
\end{proof}

Let $\mathcal{P}$ be the compatible scaling covering in Lemma \ref{lem:compatible-cover}. Similar argument as in the proof of Lemma \ref{C-cover-two} implies the following result. We omit the proof here.
\begin{lemma}\label{lem:P-cover-two}
If $U\in\mathcal{P}$ is not a singleton, then there exists $m\ge 1$ such that $\phi^m(U)$ is the union of at least two elements in $\mathcal{P}$.
\end{lemma}

Now define
$$\mathcal{A}:=\{U\cap J_K(\phi): U\in\mathcal{P}\}$$
and let $A=(A_{\gamma_i, \gamma_j})_{\gamma_i, \gamma_j\in\mathcal{A}}$ be the matrix with
$$A_{\gamma_i, \gamma_j}=
\begin{cases}
1\ \ \text{if}\ \  \gamma_j\subset\phi( \gamma_i),\\
0 \ \ \text{otherwise}.
\end{cases}$$
Denote by $(\Sigma_A,\sigma)$ the countable state Markov shift defined by $A$. We have the following code sequence for each point in $J_K(\phi)$:
\begin{lemma}\label{lem:code}
For each $x\in J_K(\phi)$, there is a unique sequence $(\gamma_j)_{j\ge 0}\in\Sigma_A$ such that $\phi^j(x)\in\gamma_j$.
\end{lemma}

By Lemma \ref{lem:code}, we obtain a coding map
$$h: J_K(\phi)\to\Sigma_A,$$
sending $x\in J_K(\phi)$ to its code sequence. Now we prove that $h$ is in fact a bijection.

\begin{proposition}\label{bijection-prefixed}
The map $h$ is a bijection from $J_K(\phi)$ onto $\Sigma_A$.
\end{proposition}
\begin{proof}
We will show that for any $(\gamma_j)_{j\ge 0}\in\Sigma_A$, there is a unique point $x\in J_K(\phi)$ such that the code sequence of $x$ is $\{\gamma_j\}_{j\ge 0}$. If there is $j_0\ge 0$ such that $\gamma_{j_0}$ is a singleton of point in $\bigcup_{n\ge0}\phi^n(\mathrm{Crit}_{J_K}(\phi))$, then we can find {such} a unique $x\in \phi^{-j_0}(\gamma_{j_0})\cap \gamma_0\subset\mathrm{GO}_K(\mathrm{Crit}^\ast_K(\phi))$. Since $\mathrm{Crit}^\ast_K(\phi)\subset J_{\C_p}(\phi)$, it follows that $x\in J_{\C_p}(\phi)$. By Theorem \ref{Thm:main1}, we conclude that $x\in J_K(\phi)$.

Now assume that $\gamma_j$ is not a singleton of point in $\bigcup_{n\ge0}\phi^n(\mathrm{Crit}^\ast_K(\phi))$ for all $j\ge 0$. Then by Lemma \ref{lem:P-cover-two}, there is a smallest $n_0\ge 1$ such that $\phi^{n_0}(\gamma_0)$ is a union of at least two elements in $\mathcal{A}$ and $\gamma_{n_0}$ is properly contained in $\phi^{n_0}(\gamma_0)$. Thus $\phi^{-n_0}(\gamma_{n_0})$ has a component $\widetilde\gamma_{n_0}$ properly contained in $\gamma_0$. Inductively, for $i\ge 1$, let $n_{i}>0$ be the smallest integer such that $\phi^{n_{i}}(\gamma_{n_{i-1}})$ is a union of at least two elements in $\mathcal{A}$. Then $\phi^{-n_{i}}(\gamma_{n_i})$ has a component $\widetilde\gamma_{n_{i}}$ properly contained in $\gamma_{n_{i-1}}$. Set $m_i:=\sum_{j=0}^in_i.$
We conclude that for $i\ge 1$, $\phi^{-m_{i}}(\gamma_{m_i})$ has a component $\tilde\gamma_{m_{i}}$ such that
$\widetilde\gamma_{m_{i}}\subsetneq \widetilde\gamma_{m_{i-1}}\subsetneq \gamma_0.$
Since the value group $|\widetilde{K}^\times|$ is discrete,
$\mathrm{diam}(\widetilde\gamma_{m_{i}})\to 0,$ as $i\to\infty$.
Then $\bigcap_{i\ge 0}\widetilde\gamma_{m_{i}}$ is a singleton. Denote the unique point in this intersection by $x$. Obviously, the code sequence of $x$ is $\{\gamma_j\}_{j\ge 0}$. We complete the proof by showing $x\in J_K(\phi)$. Indeed, since $\gamma_{m_i}\cap J_K(\phi) \not=\emptyset$, by Lemma \ref{regular-point-lemma}, there exists $x_i\in\widetilde\gamma_{m_{i}}\cap J_K(\phi)$, and observing that $x_i\to x$ as $i\to\infty$, we conclude that $x\in J_K(\phi)$, since $J_K(\phi)$ is closed.
\end{proof}

\begin{proposition}\label{conjugate-prefixed}
On $J_K(\phi)$, we have
$$h\circ\phi=\sigma\circ h.$$
\end{proposition}
\begin{proof}
Pick $x\in J_K(\phi)$ and let $(\gamma_j)_{j\ge 0}\in\Sigma_A$ be its code sequence. Then for $i\ge 0$, we have $\phi^{i}(x)\in\gamma_i$. Hence $\phi^i(\phi(x))\in\gamma_{i+1}$. Therefore the code sequence of $\phi(x)$ is $(\gamma_{j+1})_{j\ge 0}\in\Sigma_A$. The conclusion follows.
\end{proof}

Now we can prove Theorem \ref{Thm:prefixed}.
\begin{proof}[Proof of Theorem \ref{Thm:prefixed}]
By Propositions \ref{bijection-prefixed} and \ref{conjugate-prefixed}, it suffices to prove that the map $h$ is continuous and open on $I_K(\phi)$. Let $U$ be an open set in $h(I_K(\phi))$ and consider $x\in I_K(\phi)$ with $h(x)\in U$. Write $h(x)=(\gamma_i)_{i\ge 0}$. Then there exists $m\ge 1$ such that the cylinder set
$C:=[\gamma_0\cdots\gamma_m]\subset\Sigma_A$
is contained in $U$. We conclude that $x\in h^{-1}(C)$ and $h^{-1}(C)=\bigcap_{j=0}^{m}\phi^{-j}(\gamma_j)$.
Since each $\gamma_j$ is not a subset of $\mathrm{GO}_K(\mathrm{Crit}^\ast_K(\phi))$, it follows that $\gamma_j$ is an open set in $J_K(\phi)$. Note that $\phi$ is continuous. Then $\phi^{-j}(\gamma_j)$ is open in $J_K(\phi)$. Moreover, $\phi^{-j}(\gamma_j)\cap J_K(\phi)\not=\emptyset$. {Then} $h^{-1}(C)\cap J_K(\phi)$ is open. Therefore, $h$ is continuous.

Now let $W$ be an open subset in $I_K(\phi)$. We claim that there exists $\ell\ge 1$ such that $\phi^\ell(W)$ contains an element in $\mathcal{A}$. Indeed, for otherwise, there exists $\gamma_j\in\mathcal{A}$ such that $\phi^j(W)\subset\gamma_j$ for all $j\ge 0$, then $W$ is a subset of $\bigcap_{j=0}^\infty\phi^{-j}(\gamma_j)$, and hence by Lemma \ref{lem:P-cover-two} and the local compactness of $K$, the set $W$ is a singleton, which is a contradiction since $W$ is open. Now pick $(\gamma_i)_{i\ge 0}\in h(W)$. Note that $\phi$ is a bijection from $\gamma$ onto $\phi(\gamma)$ for any $\gamma\in\mathcal{A}$. Then
the cylinder set $[\gamma_0\cdots\gamma_\ell]\subset\Sigma_A$ is contained in $h(W)$.
Therefore, $h$ is open.
\end{proof}

\subsection{Proof of Theorem \ref{Thm:GF}}\label{sec:general}
In this subsection, we prove Theorem \ref{Thm:GF}. We first show the following result concerning the rational maps whose iteration has iteratedly prefixed Julia critical points.

\begin{proposition}\label{prefixed implies conjugate}
Let $\phi\in K(x)$ be a rational map of degree at least $2$. Suppose there exists $n\ge 1$ such that the critical points of $\phi^n$ in $J_{\C_p}(\phi)\cap\mathbb{P}^1_K$ are iteratedly prefixed. Then the conclusion of Theorem \ref{Thm:GF} holds.
\end{proposition}

\begin{proof}
From the proof of Theorem \ref{Thm:prefixed}, there exist a partition $\mathcal{A}$ of $J_K(\phi^n)$ and an $|\mathcal{A}|\times |\mathcal{A}|$ matrix $A$ such that the dynamical system $(I_K(\phi^n), \phi^n)$ is conjugate to a countable state Markov shift $(\Sigma_A,\sigma)$, where $\Sigma_A\subset\mathcal{A}^\mathbb{N}$ is the subshift space defined by $A$. Note that by Theorem \ref{Thm:main1}, we have $J_K(\phi)=J_K(\phi^n)$. Then $\mathcal{A}$  is, in fact, a partition of $J_K(\phi)$. It follows that the map
$H: J_K(\phi)\to\mathcal{A}^\mathbb{N},$
sending $x\in J_K(\phi)$ to $(\gamma_0\gamma_1\cdots)\in\mathcal{A}^\mathbb{N}$ with $\phi^{i}(x)\in \gamma_i$, is well-defined.
Now we show that $H$ is injective. Indeed, considering the map
$F:\Sigma_A\to\mathcal{A}^\mathbb{N},$
sending $(\gamma_0\gamma_1\gamma_2\cdots)$ to $(\gamma_0\gamma_n\gamma_{2n}\cdots)$, and letting $h$ be as in Theorem \ref{Thm:prefixed}, we have
$h=F\circ H$; moreover,  if $x_1, x_2\in J_{\mathbb{P}^1_K}(\phi^n)$ satisfy $H(x_1)=H(x_2)$, then  $h(x_1)=h(x_2)$ and hence $x_1=x_2$.

It is obvious that $H\circ\phi=\sigma\circ H$. Now we show that $H$ is a homeomorphism on $I_K(\phi)$. By Theorem \ref{Thm:main1}, we obtain $I_K(\phi)=I_K(\phi^n)$.
Hence $h$ is a homeomorphism on $I_K(\phi)$. Then the openness of $H$ follows from the continuity of $F$ and the fact that away from $H(\mathrm{GO}_K(\mathrm{Crit}^\ast_K(\phi)))$ the map $F$ is locally one-to-one. The continuity of $H$ follows from the openness of $F$ since $F(U)=h(H^{-1}(U))$ for any set $U\subset H(J_K(\phi))$.
\end{proof}

Now we can finally prove Theorem \ref{Thm:GF}.
\begin{proof}[Proof of Theorem \ref{Thm:GF}]
Since $\phi\in K(z)$ is a sub-hyperbolic rational map, the points in $\mathrm{Crit}^\ast_K(\phi)$ are mapped to repelling cycles. Since $\mathrm{Crit}^\ast_K(\phi)$ is a finite set, let $n$ be the least common multiple of the periods of these repelling cycles. Then $\phi^n$ maps each point in  $\mathrm{Crit}^\ast_K(\phi)$ to a repelling fixed point of $\phi^n$. By Theorem \ref{Thm:main1}, in $\mathbb{P}^1_K$, the Julia critical points of $\phi^n$ are the Julia critical points of $\phi$ and their $\ell$-th preimages for $1\le \ell<n$. Hence the Julia critical points of $\phi^n$ are iteratedly prefixed. Then Theorem \ref{Thm:GF} immediately follows from Theorem \ref{Thm:prefixed} and Proposition \ref{prefixed implies conjugate}.
\end{proof}

\subsection{Applications of Theorem \ref{Thm:GF}}\label{sec:applications}

Theorem \ref{Thm:GF} shows that any sub-hyperbolic rational map on its Julia set omitting the grand orbits of critical points is topologically conjugate to a countable state Markov shift with a matrix $A$. As in Section \ref{sec:Gurevich}, we consider the natural directed graph associated with $A$, and find its irreducible components. For one such irreducible component $A'$, the subsystem $(\Sigma_{A'}, \sigma)$ is transitive and we can obtain a series
$G_\alpha(z)=\sum_{n\geq1} \Xi_{\alpha}(n)z^n$ by choosing one state $\alpha$ and by counting the numbers $\Xi_{\alpha}(n)$ of first return loops at $\alpha$.

Let $\log \lambda$ denote the Gurevich entropy of the subshift  $(\Sigma_{A'}, \sigma)$. Recall that the subshift $(\Sigma_{A'}, \sigma)$ is {\it positive recurrent} if 
\[ \sum_{n=1}^\infty {\Xi_\alpha(n) \over \lambda^n}=1 \quad \text{and} \quad \sum_{n=1}^\infty {n\Xi_\alpha(n)\over \lambda^n}<\infty,\]
and is {\it strongly positive recurrent} if 
\[ \limsup_{n\to\infty} \Xi_\alpha(n)^{1/n} < \lambda.\]
 A result of Gurevich \cite{Gurevivc70, Gurevich98} 
implies  that $(\Sigma_{A'}, \sigma)$ is positive recurrent if and only if there is a unique measure of maximal entropy on $(\Sigma_{A'}, \sigma)$; and a result of Boyle, Buzzi and G\'omez \cite{Boyle06}
implies that $(\Sigma_{A'}, \sigma)$ is strongly positive recurrent if and only if $(\Sigma_{A'}, \sigma)$ has a unique maximal entropy measure and with this measure the subshift is exponentially mixing.

Hence, by examining the corresponding matrix $A$, one can not only calculate the Gurevich entropy of a sub-hyperbolic rational map and its transitive subsystems, but also determine if on the transitive subsystems there exists a unique measure of maximal entropy and further if this maximal entropy measure is exponentially mixing.  


 
 \subsection{Proof of Theorem \ref{globaldynamics}}\label{sec:pf-global}
 By Theorem \ref{Thm:GF}, we only need to consider the dynamics on the $F_K(\phi)$. Since  $F_{\C_p}(\phi)\cap \P^1_K=F_K(\phi)$, by Theorem \ref{Thm:main1}, it suffices to consider the components in  $F_{\C_p}(\phi)$ intersecting $\mathbb{P}^1_K$. 
 Let $U\subset F_{\mathbb{C}_p}(\phi)$ be such a component. It follows from Theorem \ref{no-Wandering-thm} that $U$ is (eventually) periodic. Considering iterated images if necessary,  we can assume that $U$ is periodic under $\phi$. Then by Corollary \ref{Classification-theorem} and Theorem \ref{no-Wandering-thm}, we conclude that $U$ is either an attracting component or an indifferent component, and there are only finitely many such $U$. If $U$ is an attracting component, then all points in $U$ converge to an attracting periodic orbit under iteration of $\phi$, moreover, the completeness of $K$ implies that this attracting periodic orbit is in $\mathbb{P}^1_K$.
 
 Now  we investigate the case that $U$ is a periodic indifferent component.  Considering iteration of $\phi$ if necessary, we can further assume that $U$ is fixed by $\phi$. By \cite[Theorem 9.7]{Benedetto19}, the component $U$  is a connected affinoid, that is, we can write $U$ as the complement of the union of finitely many closed $\P^{1}_{\C_p}$-disks.
Denote by $U_K:=U\cap \P^1_K$. Observing that 
 $U_K$ is compact and open, we can write  $U_K$ as the union of finitely many  $\P^{1}_K$-disks $B_j$ with radius less than $1$. Since $\phi$ has coefficients in $K$ and hence fixes $U_K$, we conclude that  there is $\ell\ge 1$ such that  $\phi^{\ell}$ fixes each $B_j$, and furthermore, $\phi^\ell$ is an isometry on each $B_j$.   Pick  a M\"obius transformation $M\in \mathrm{PGL}(2,K)$ mapping $B_j$ to the ring  $\mathcal{O}_K$ of integers of $K$, and write $M\circ \phi^\ell\circ M^{-1}$ as Taylor expansion 
 \[M\circ \phi^\ell\circ M^{-1}(x)= \sum_{i=0}^{\infty} a_ix^i, \quad  a_i\in \mathcal{O}_K. \] 
 We have that $|a_0|_p\leq 1, |a_1|_p=1$  and $|a_i|_p<1$ for $i\geq 2$. Applying \cite[Theorem 1.1]{Fan15}, we conclude that  $\mathcal{O}_K=M(B_j)$ and hence $B_j$ can be decomposed into two parts: one is the set of indifferent periodic points and the other is the union of all (at most countably
many) clopen invariant sets each of which is a finite union
of $\P^1_K$-disks and can be decomposed to minimal subsystems. Since there are finitely many such $U$ and $B_j$, we complete the proof.

\bibliographystyle{acm}

\end{document}